\documentclass{article}

\usepackage{iclr2027_conference,times}
\usepackage{amsmath,amssymb,amsthm,mathtools}
\usepackage{booktabs}
\usepackage{graphicx}
\usepackage{microtype}
\usepackage{enumitem}
\usepackage{array}
\usepackage{algorithm}
\usepackage{algpseudocode}
\usepackage{placeins}
\usepackage{tikz}
\usetikzlibrary{arrows.meta}
\usepackage{hyperref,url}
\hypersetup{hidelinks,pdfauthor={},pdftitle={No Universal Remainder Rate for Chambolle--Dossal Acceleration}}
\usepackage[nameinlink,capitalise,noabbrev]{cleveref}

\allowdisplaybreaks

\newtheorem{theorem}{Theorem}

\newtheorem{lemma}[theorem]{Lemma}

\theoremstyle{definition}

\newcommand{\R}{\mathbb R}

\newcommand{\dist}{\operatorname{dist}}
\newcommand{\proofstep}[2]{\par\addvspace{\medskipamount}\noindent\emph{#1. #2}\enspace\ignorespaces}

\title{No Universal Remainder Rate for\\Chambolle--Dossal Acceleration}
\usepackage{etoolbox}
\iclrfinalcopy
\makeatletter
\patchcmd{\@maketitle}
  {Published as a conference paper at ICLR 2027}
  {Preprint}
  {}
  {\PackageError{preprint}{Cannot replace the conference header}
    {Check the title macro before compiling the public version.}}
\g@addto@macro\@thanks{%
  \footnotetext[3]{Xinan Dai and Wenhao Deng contributed to this research
    during their internships at Westlake University.}%
}
\makeatother

\hypersetup{pdfauthor={Yuchen Yang, Xinan Dai, Wenhao Deng, Yingdong Shi, Feng Xu, Tailin Wu}}
\newcommand{\preprintauthorblock}{%
  \begin{minipage}[t]{\dimexpr\textwidth-2\tabcolsep\relax}
  \raggedright
  Yuchen Yang\textsuperscript{1,*}\hspace{0.7em}%
  Xinan Dai\textsuperscript{1,2,*,\ensuremath{\ddagger}}\hspace{0.7em}%
  Wenhao Deng\textsuperscript{1,3,\ensuremath{\ddagger}}\par\vspace{0.3ex}
  Yingdong Shi\textsuperscript{4}\hspace{0.7em}%
  Feng Xu\textsuperscript{5}\hspace{0.7em}%
  Tailin Wu\textsuperscript{1,\ensuremath{\dagger}}
  \par\vspace{0.6ex}
  \normalfont\small
  \textsuperscript{1}Department of Artificial Intelligence, School of Engineering,
  Westlake University, Hangzhou, China\par
  \textsuperscript{2}Key Laboratory for Information Science of Electromagnetic Waves,
  College of Future Information Technology, Fudan University, Shanghai, China\par
  \textsuperscript{3}University of Glasgow, Glasgow, United Kingdom\par
  \textsuperscript{4}School of Information Science and Technology,
  ShanghaiTech University, Shanghai, China\par
  \textsuperscript{5}School of Information Science and Technology,
  Fudan University, Shanghai, China\par
  \vspace{0.6ex}
  \textsuperscript{*}Equal contribution.\quad
  \textsuperscript{\ensuremath{\dagger}}Corresponding author:
  \href{mailto:wutailin@westlake.edu.cn}{\texttt{wutailin@westlake.edu.cn}}.
  \end{minipage}%
}
\author{\preprintauthorblock}

\date{}

\begin{document}
\maketitle

\begin{abstract}
Chambolle--Dossal acceleration guarantees $F(x_n)-F^*=o(n^{-2})$ for
every fixed smooth convex loss with a minimizer. We show that this
qualitative improvement admits no universal quantitative rate. For every
damping parameter $\alpha>3$ and positive nondecreasing gain
$G(n)\to\infty$, we construct a fixed one-dimensional smooth convex
loss whose exact CD orbit satisfies
\[
 \sup_{n\ge1} n^2G(n)\bigl(F(x_n)-F^*\bigr)=\infty.
\]
Thus no divergent gain improves the $n^{-2}$ scale for all fixed losses,
even with instance-dependent constants. The construction prescribes queried
gradients and realizes infinitely many slow blocks within one smooth
convex objective. Under local $p$-power growth with $p>2$ and sufficiently
strong damping, we also construct a fixed loss whose exact CD orbit satisfies
$F(x_n)-F^*\sim Dn^{-2p/(p-2)}$, $D>0$, establishing the sharpness of
the known convergence rate. Both main results are formally verified in Lean~4.
\end{abstract}

\section{Introduction}

The accelerated $O(n^{-2})$ function-value guarantee has a strict
fixed-instance refinement: Chambolle--Dossal (CD) acceleration achieves
$o(n^{-2})$ on every fixed smooth convex objective with a minimizer
\citep{Chambolle2015,AP16}. This means that the normalized residual
$n^2(F(x_n)-F^*)$ tends to zero. Does this vanishing factor admit any
common quantitative decay rate, with constants allowed to depend on the
objective and initialization? We show that it does not, even for
one-dimensional coercive losses with a prescribed smoothness budget and
initial distance.

Consider CD acceleration with damping $\alpha>3$.
For a convex loss $F$ with globally $L$-Lipschitz gradient and a fixed step
$0<s<L^{-1}$, its update is
\begin{equation}
\begin{aligned}
y_n&=x_n+\beta_n(x_n-x_{n-1}),&x_{n+1}&=y_n-s\nabla F(y_n),\\
\beta_n&=\frac{n-1}{n+\alpha-1},&n&\ge1.
\end{aligned}
\label{eq:algorithm-intro}
\end{equation}
The normalized residual $q_n=n^2(F(x_n)-F^*)$ tends to zero on every
fixed loss. A universal remainder rate would give $q_n=O(1/G(n))$ for
some common positive nondecreasing $G(n)\to\infty$, with
instance-dependent constants. For each such $G$, we construct one fixed
loss with $\sup_n G(n)q_n=\infty$. Thus no common gain exists, however
slowly it grows.

We prescribe nonincreasing queried gradients and invert the CD recurrence
to recover an exact orbit, using separated constant-gradient blocks to
defeat the proposed gain along a subsequence. The inverse construction
and choice of scales have precedents in scalar gradient flow
\citep[Lemma~2.7 and Example~2.9]{siegel2023qualitativedifferencegradientflows}. We impose weighted summability
to ensure convergence and construct the loss by interpolating its
derivative, with bounded slopes enforcing global smoothness
\citep[see also][for the general finite-data interpolation framework]{Taylor2017}.
The resulting loss is one-dimensional, even, coercive, and uniquely
minimized at zero; spatial rescaling sets any prescribed positive initial
distance without changing the gradient Lipschitz constant.

Using the same inverse-orbit construction with gradients that decay as a
power, we also obtain fixed losses whose exact CD orbits satisfy
$F(x_n)-F^*\sim Dn^{-2p/(p-2)}$, $D>0$. These instances establish the
sharpness of the known $O(n^{-2p/(p-2)})$ rate under local $p$-power
growth with suitable flatness or sufficiently strong damping.

We establish the following two sharpness results for CD acceleration.
Both conclusions are formally verified in Lean~4 \citep{10.1007/978-3-030-79876-5_37}; see the supplementary
materials in Appendix~\ref{app:supplement}.
\begin{itemize}
\item \textbf{No universal remainder rate on fixed losses.}
For every positive nondecreasing $G(n)\to\infty$, we construct one
fixed one-dimensional smooth convex loss whose exact CD orbit satisfies
\[
 F(x_n)-F^*=o(n^{-2}),\qquad
 \sup_{n\ge1}n^2G(n)\bigl(F(x_n)-F^*\bigr)=\infty.
\]
The smoothness constant and initial distance can be prescribed, and the
construction extends to every fixed finite dimension. To our knowledge,
this is the first no-universal-gain result for exact discrete CD iterates
on a single fixed scalar loss (\cref{thm:no-gain}).
\item \textbf{Sharp rates under local geometry.}
For every admissible parameter choice in \cref{thm:power}, we construct
fixed smooth convex losses with local $p$-power growth whose exact CD
orbits satisfy
\[
 F(x_n)-F^*\sim Dn^{-2p/(p-2)},\qquad D>0.
\]
These instances attain the established upper rates
\citep{Apidopoulos2021,aujol2024strongconvergencefistaiterates}. To our knowledge, they are the first exact
discrete constructions proving sharpness throughout both stated geometry
regimes, including endpoint flatness.
\end{itemize}

\section{Related work}\label{sec:related}

Our results connect four strands of work: accelerated convergence
guarantees, finite-horizon lower bounds and interpolation, slow convergence
on fixed objectives, and rates under local geometry. We distinguish
horizon-dependent worst cases from one fixed loss, and continuous
trajectories from exact discrete orbits.

\subsection{Accelerated convergence guarantees}

Nesterov's accelerated methods achieve the classical $O(n^{-2})$ bound
\citep{Nesterov04}, which FISTA extends to composite convex minimization
\citep{BT09}. \citet{Chambolle2015} modify the momentum schedule to establish
convergence of the iterates; \citet{AP16} prove the strict $o(n^{-2})$
function-value rate for $\alpha>3$. Weighted residual summability,
$\sum_n n(F(x_n)-F^*)<\infty$, is also established for this discrete
schedule \citep{Chambolle2015,JMLR:v17:15-084}. Neither little-$o$ convergence nor this
summability bound supplies a common eventual pointwise remainder rate.
Excluding every improved polynomial exponent would still be insufficient:
$1/(n^2\log^2(n+1))$, for example, admits a logarithmic gain while
failing every $O(n^{-2-\varepsilon})$ bound. \Cref{thm:no-gain} excludes
every prescribed divergent gain using one scalar loss, even when the
bound's constant and starting index may depend on that loss.

\subsection{Finite-horizon lower bounds and interpolation}

Oracle complexity gives worst-case lower bounds for prescribed iteration
budgets \citep{Nesterov04}. For oblivious first-order canonical linear
iterative methods, whose coefficients may depend on time and supplied
smoothness and strong convexity parameters, \citet{pmlr-v48-arjevani16} obtain
lower bounds valid in every fixed dimension.
Performance estimation optimizes over objectives and oracle data at a
given horizon \citep{Drori2014}; smooth convex interpolation makes this
approach constructive and yields exact worst-case analyses
\citep{taylor2017exact,Taylor2017}.

These finite-horizon lower bounds and performance-estimation analyses
allow the worst-case objective to depend on the budget. They do not by
themselves give $\forall G\,\exists F_G$ with unbounded
$n^2G(n)(F_G(x_n)-F_G^*)$ along one infinite CD orbit.
Interpolation provides an established criterion for compatibility of
first-order data. Our task is to design one infinite scalar sequence
that obeys the exact CD recurrence and produces the required slow blocks:
the query points remain ordered and accumulate at the minimizer, while
all interpolation slopes obey the prescribed Lipschitz bound.
Independently chosen finite-horizon worst-case traces need not satisfy
these joint requirements.

\subsection{Slow convergence on fixed objectives}

The closest antecedent is \citet{siegel2023qualitativedifferencegradientflows}. Their scalar inverse construction
and sparse-scale example \citep[Lemma~2.7 and Example~2.9]{siegel2023qualitativedifferencegradientflows} show
that gradient flow admits no common divergent gain over its $1/t$
scale, even in one dimension. Their slow inertial construction with an
attained minimum uses an infinite-dimensional quadratic objective
\citep[Lemma~4.2]{siegel2023qualitativedifferencegradientflows}.
These results already establish the principles of prescribing a
trajectory and placing slow portions at widely separated scales.

These constructions do not supply a scalar objective for the exact CD
recurrence. For damping $3/t$, the ODE approximation of
\citet[Theorem~2]{JMLR:v17:15-084} takes the step size to zero on bounded time
intervals; it does not transfer infinite-time lower asymptotics to a
fixed step size. We prescribe the queried gradients and solve the CD
velocity recurrence exactly. The compatibility estimates
combine weighted summability, $\sum_n n g_n<\infty$, with a lower bound
on the loss during each constant-gradient block and control of the
interpolation slopes. The automatic slope bound $1/s$ is insufficient
when $s<1/L$; the block construction enforces the smaller prescribed
bound $L$. These estimates allow one loss to retain infinitely many slow
blocks while its orbit converges to its unique minimizer
(\cref{sec:construction}).

\subsection{Local geometry and matching power rates}

Flatness and growth conditions quantify how the objective behaves near
its minimizers. For continuous inertial dynamics, power potentials with $p>2$
already give matching trajectories at the $t^{-2p/(p-2)}$ scale
\citep{Attouch2018}, and \citet{ADR19} develop sharp geometry-dependent ODE
rates. In discrete time, \citet{Apidopoulos2021} establish
$O(n^{-2p/(p-2)})$ under local $p$-power growth, suitable flatness,
uniqueness of the minimizer, and a damping threshold.
They also present numerical evidence for sharpness using power
potentials, while leaving discrete optimality unproved.
\citet[Theorem~1]{aujol2024strongconvergencefistaiterates} obtain the same exponent for coercive
objectives under local growth alone when $\alpha>5+8/(p-2)$, without
requiring flatness or a unique minimizer.

These upper bounds underlie \cref{thm:power}, but neither they nor the
continuous power trajectories supply an exact discrete attaining instance.
Our construction prescribes
$g_n=(n+M)^{-\rho}$, with $\rho=2+2/(p-2)$, and recovers one fixed
smooth loss with exact asymptotics
$F(x_n)-F^*\sim Dn^{-2p/(p-2)}$, $D>0$.
An additional step establishes endpoint flatness: the leading behavior
$F(x)\sim C|x|^p$ alone does not imply $xF'(x)\ge p(F(x)-F^*)$.
We choose the shift $M$ to give the first asymptotic correction the
required sign and control interpolation errors between query points
(Appendix~\ref{app:power}), yielding matching instances in both stated
regimes, including the flatness endpoint.

\section{Setting and main results}\label{sec:results}

\subsection{Objectives and the CD iteration}

We consider convex differentiable losses $F:\R^d\to\R$, in a fixed
finite dimension $d\ge1$, that attain a minimum $F^*$ and have a globally
$L$-Lipschitz gradient:
\[
 \|\nabla F(x)-\nabla F(z)\|\le L\|x-z\|
 \qquad\text{for all }x,z\in\R^d.
\]
Here $L>0$, and norms and inner products are Euclidean. We use $C^{1,1}$
for differentiability with a globally Lipschitz gradient, and $F'$ for the
derivative in dimension one. A loss is coercive if $F(x)\to\infty$ as
$\|x\|\to\infty$.

For damping $\alpha>3$, a fixed step $0<s<L^{-1}$, and initialization
$x_0,x_1\in\R^d$, the Chambolle--Dossal (CD) iteration is
\begin{equation}
 \begin{aligned}
 y_n&=x_n+\beta_n(x_n-x_{n-1}),&
 x_{n+1}&=y_n-s\nabla F(y_n),\\
 \beta_n&=\frac{n-1}{n+\alpha-1},& n&\ge1.
 \end{aligned}
 \label{eq:algorithm}
\end{equation}
Here $x_n$ is the iterate, $\beta_n$ the
momentum coefficient, and $y_n$ the extrapolated query point. Write
$g_n:=\nabla F(y_n)$ for the queried gradient. The objective error is
$r_n:=F(x_n)-F^*$ and the normalized residual is $q_n:=n^2r_n$.
Since $\beta_1=0$,
$x_0$ does not affect later iterates; our constructions use $x_0=x_1$.
All convergence rates are as $n\to\infty$.

\subsection{No universal remainder rate}

A gain is a positive nondecreasing sequence $G(n)\to\infty$.
A universal remainder rate would give $q_n=O(1/G(n))$ on every fixed
loss satisfying the assumptions above, for every CD initialization.
Both the constant and the starting index may depend on the loss and
initialization. The following theorem excludes such a rate.

\begin{theorem}[No universal gain]\label{thm:no-gain}
Fix $\alpha>3$, $L>0$, $0<s<L^{-1}$, an initial distance $R>0$,
and a gain $G$.
There is a fixed even, coercive convex loss $F:\R\to\R$ with an
$L$-Lipschitz derivative, uniquely minimized at zero with $F(0)=0$,
whose exact CD orbit from $x_0=x_1=R$ satisfies
\begin{equation}
 \sup_{n\ge1}n^2G(n)\bigl(F(x_n)-F^*\bigr)=\infty.
 \label{eq:no-gain}
\end{equation}
The construction extends to every fixed finite dimension, preserving
smoothness, coercivity, the unique minimizer, initial distance, and
objective errors.
\end{theorem}

For each proposed $G$, one loss remains fixed along the entire orbit.
Its normalized residual satisfies $q_n\to0$ by the established
little-$o$ guarantee, while $G(n)q_n$ is unbounded along a subsequence.
Thus even an eventual bound with an instance-dependent constant fails.
The loss may depend on $G$: taking $G(n)=n^\varepsilon$, $\varepsilon>0$,
excludes any common improved polynomial exponent, and
$G(n)=\log(n+1)$ excludes a common logarithmic gain. The theorem also
covers arbitrarily slower divergent gains.

\subsection{Local geometry and sharp power rates}

Suppose first that $F$ has a unique minimizer $x^*$. For $\beta>2$, the
local flatness condition $H(\beta)$ means that, for some $\eta>0$,
\begin{equation}
 H(\beta):\qquad F(x)-F^*\le\frac1\beta\langle\nabla F(x),x-x^*\rangle
 \qquad\text{if }\|x-x^*\|\le\eta.
 \label{eq:H}
\end{equation}
For $p>2$, the local growth condition $L(p)$ means that, for some
$K,\varepsilon>0$,
\begin{equation}
 L(p):\qquad F(x)-F^*\ge K\|x-x^*\|^p
 \qquad\text{if }\|x-x^*\|\le\varepsilon.
 \label{eq:Lp}
\end{equation}
Both inequalities apply to every $x$ in the indicated neighborhood.
Here $\beta$ is a flatness exponent, distinct from the momentum
coefficient $\beta_n$.

To allow multiple minimizers, let $X^*$ denote the minimizers of $F$
and $\dist(x,X^*)$ the Euclidean distance from $x$ to the nearest
minimizer. For $p>2$, the local growth condition $L_{\mathrm{set}}(p)$
means that, for some $K,\varepsilon>0$,
\begin{equation}
 L_{\mathrm{set}}(p):\qquad F(x)-F^*\ge K\dist(x,X^*)^p
 \qquad\text{if }\dist(x,X^*)\le\varepsilon.
 \label{eq:Gp}
\end{equation}
The constants apply throughout this neighborhood of the minimizers.
When the minimizer is unique, $L_{\mathrm{set}}(p)$ reduces to $L(p)$.
\begin{theorem}[Sharp power rates under local geometry]\label{thm:power}
Fix $\alpha>3$, $L>0$, $0<s<L^{-1}$, and $p>2$. Let
$F:\R^d\to\R$ be convex, attain its minimum, and have an
$L$-Lipschitz gradient. For every CD initialization,
\begin{equation}
 F(x_n)-F^*=O\bigl(n^{-2p/(p-2)}\bigr)
 \label{eq:power-upper}
\end{equation}
holds under either of the following conditions:
\begin{enumerate}[label=(\roman*),leftmargin=*]
 \item $F$ has a unique minimizer and satisfies $H(\beta)$ and $L(p)$,
 with $2<\beta\le p$ and $\alpha\ge\frac{\beta+2}{\beta-2}$.
 \item $F$ satisfies $L_{\mathrm{set}}(p)$ and
 $\alpha>5+8/(p-2)$.
\end{enumerate}
In each regime, for every admissible parameter choice and every fixed
finite dimension, there is a coercive loss satisfying these assumptions
and having a unique minimizer, whose exact CD orbit for some
initialization satisfies
\begin{equation}
 F(x_n)-F^*\sim Dn^{-2p/(p-2)},\qquad D>0.
 \label{eq:power-sharp}
\end{equation}
Thus the upper rate is attained with a positive asymptotic constant
$D$, and no common divergent gain improves it. The constants and onset
of the upper bound may depend on the loss and initialization.
\end{theorem}

The upper bounds in \cref{thm:power} build on established results.
\citet{Apidopoulos2021} prove the rate in part~(i) under flatness and growth
conditions. \citet{aujol2024strongconvergencefistaiterates} obtain the growth-only rate used in part~(ii)
with an additional coercivity assumption. We extend their upper bound
to the finite-dimensional setting of part~(ii) without requiring
coercivity.

Our contribution is a matching construction for every parameter choice
allowed by the theorem: one fixed smooth convex loss has an exact CD
orbit with $F(x_n)-F^*\sim Dn^{-2p/(p-2)}$, $D>0$.
These instances prove that the upper rate is sharp in both regimes.

As $p\to\infty$, the exponent $2p/(p-2)$ approaches $2$.
The case $p=2$ lies outside this power law.

\section{Constructing the fixed losses}\label{sec:construction}

This section constructs the fixed losses in \cref{thm:no-gain} and the
sharpness instances in \cref{thm:power}. The construction proceeds in
three steps: prescribe a sequence of queried gradients, recover the
corresponding CD orbit, and interpolate and integrate the prescribed
derivative values to obtain a smooth convex loss.
\Cref{sec:inverse-framework} develops this common procedure and explains
how it gives the required properties. \Cref{sec:blocks} uses separated
constant-gradient blocks to obtain the no-universal-gain instance, while
\cref{sec:power} uses a shifted power sequence to obtain a loss with the
exact power-rate asymptotic.

\subsection{Recovering the orbit and the loss}\label{sec:inverse-framework}

We reconstruct a fixed loss from a prescribed gradient sequence so that
the CD algorithm follows the designed orbit exactly. Choose $(g_n)_{n\ge1}$
with
\[
 g_n>0,\qquad g_{n+1}\le g_n,\qquad
 \sum_{n=1}^{\infty}ng_n<\infty.
\]
Write $v_n=x_{n-1}-x_n$ for the displacement towards zero. Recover the
displacements, iterates, and query points by
\begin{equation}
\begin{aligned}
 v_{n+1}&=\beta_nv_n+sg_n, & v_1&=0,\\
 x_n&=\sum_{j=n+1}^{\infty}v_j, & x_0&=x_1,\qquad
 y_n=x_n-\beta_nv_n .
\end{aligned}
\label{eq:inverse}
\end{equation}
These sequences satisfy
\[
 v_{n+1}>0,\qquad \sum_{n=1}^{\infty}v_{n+1}<\infty,
 \qquad x_n\downarrow0,\qquad y_n\downarrow0,
\]
with $x_{n+1}=x_n-v_{n+1}$ and $y_n-sg_n=x_{n+1}$.
Thus the total displacement is finite, and the query points form an
ordered mesh accumulating at zero.

Join consecutive points $(y_n,g_n)$ by line segments to define the
derivative on the positive half-axis, including its extension at zero
and beyond the first query:
\begin{equation}
 \psi(x)=
 \begin{cases}
  0, & x=0,\\[2pt]
  g_{n+1}+\dfrac{g_n-g_{n+1}}{y_n-y_{n+1}}(x-y_{n+1}),
    & y_{n+1}\le x\le y_n,\quad n\ge1,\\[5pt]
  g_1, & x\ge y_1.
 \end{cases}
 \label{eq:derivative-interpolant}
\end{equation}
Finally, define the loss on the whole real line by
\begin{equation}
 F(x)=\int_0^{|x|}\psi(u)\,du.
 \label{eq:integrated-loss}
\end{equation}
The resulting loss is even, coercive, convex, and $C^{1,1}$, with
unique minimizer zero. Its derivative is $L$-Lipschitz whenever
\[
 0\le\frac{g_n-g_{n+1}}{y_n-y_{n+1}}\le L
 \qquad(n\ge1).
\]
Each gradient schedule below enforces this bound. Moreover,
$F'(y_n)=g_n$ and $x_{n+1}=y_n-sF'(y_n)$, so the prescribed points
form an exact CD orbit of this one fixed loss.
All these properties are proved in \cref{lem:inverse} in
Appendix~\ref{app:inverse}.

\subsection{Theorem 1: making slow episodes compatible with convergence}
\label{sec:blocks}

For a prescribed gain $G$, we construct increasingly long gradient
plateaus with decreasing heights. Lowering the heights keeps the total
displacement finite, while lengthening the plateaus retains enough error
to violate the proposed rate $1/(n^2G(n))$ at arbitrarily late iterates.
\Cref{fig:block-construction} summarizes how these two requirements
are met within one gradient sequence.

\begin{figure}[!htb]
\centering
\includegraphics[width=\linewidth]{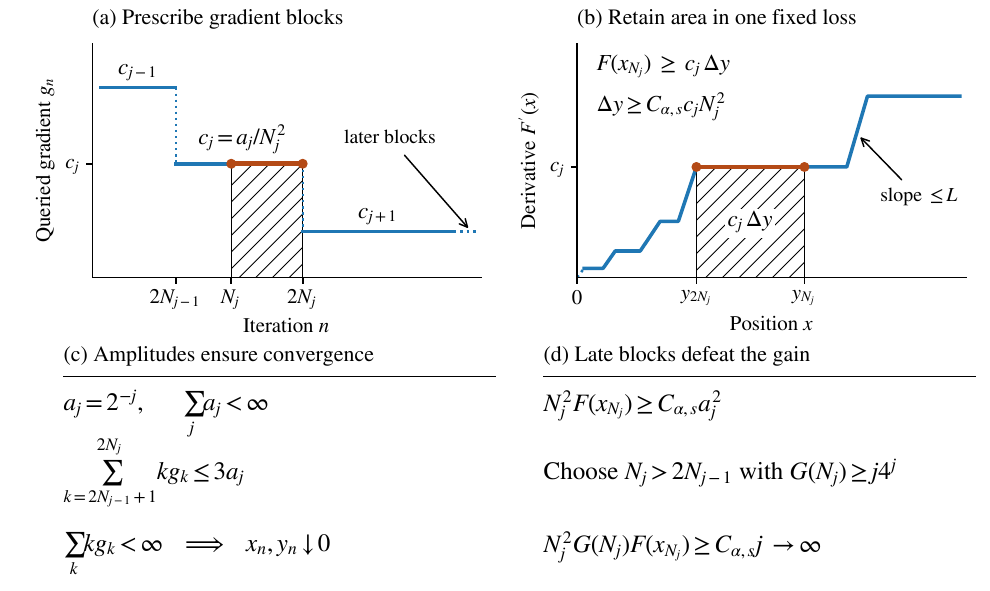}
\caption{Block construction; axes are not to scale. Orange marks the same
constant-gradient interval in time (a) and space (b), retaining area
$c_j\Delta y\le F(x_{N_j})$. Summable amplitudes ensure convergence (c);
late blocks amplify the residual (d). $C_{\alpha,s}>0$ is the block constant.}
\label{fig:block-construction}
\end{figure}

First partition the positive iteration indices into consecutive blocks.
Set $a_j=2^{-j}$ and $N_0=0$, and choose integers $N_j\ge8$ with
$N_j>2N_{j-1}$; their dependence on $G$ will be specified below.
On each block, prescribe the constant queried gradient
\begin{equation}
 g_k=c_j=\frac{a_j}{N_j^2}
 \qquad\text{for }2N_{j-1}<k\le2N_j.
 \label{eq:block-schedule}
\end{equation}
These blocks cover all positive iteration indices, and the heights $c_j$
decrease from block to block. Thus the gradient sequence is positive
and nonincreasing. Each plateau consists of equal values of
$g_k=F'(y_k)$; the loss values continue to decrease along the orbit.

To preserve convergence as the blocks grow longer, their heights decrease
in proportion to $N_j^{-2}$. A block contains $O(N_j)$ indices with
weights $k=O(N_j)$. More precisely, its weighted gradient sum satisfies
\[
 \sum_{k=2N_{j-1}+1}^{2N_j}kg_k
 \le3c_jN_j^2=3a_j.
\]
Since $\sum_j a_j<\infty$, this gives $\sum_k kg_k<\infty$.
The inverse construction therefore has finite total displacement and
$x_n,y_n\downarrow0$, as established in \cref{lem:inverse}.
Crucially, each block's contribution is controlled by $a_j$ independently
of $N_j$: we can place a block arbitrarily late without increasing
this bound.

At the same time, a long plateau retains a lower bound on the error.
The separation $N_j>2N_{j-1}$ places $[N_j,2N_j]$ entirely within
block $j$. Over this interval, momentum accumulates the repeated
gradient contributions and produces a spatial width of at least
$C_{\alpha,s}c_jN_j^2$, where $C_{\alpha,s}>0$ depends only on the
damping and step. The interpolated derivative equals $c_j$ on
$[y_{2N_j},y_{N_j}]$, so the area under this segment gives
\[
 F(x_{N_j})\ge c_j\bigl(y_{N_j}-y_{2N_j}\bigr)
 \ge C_{\alpha,s}c_j^2N_j^2
 =\frac{C_{\alpha,s}a_j^2}{N_j^2}.
\]
\Cref{lem:block} proves this estimate independently of the gradients
before the block. Hence the normalized residual satisfies
$N_j^2F(x_{N_j})\ge C_{\alpha,s}a_j^2$ at the selected indices.
This is a subsequence lower bound; its vanishing coefficient $a_j^2$
is consistent with the strict $o(n^{-2})$ guarantee.

Finally, choose the blocks late enough to defeat the proposed gain.
Since $G(n)\to\infty$,
each $N_j$ can be chosen large enough that $G(N_j)\ge j4^j$, in addition
to the separation requirements. Multiplying the retained lower bound by
this gain gives
\begin{equation}
 N_j^2G(N_j)F(x_{N_j})
 \ge C_{\alpha,s}G(N_j)a_j^2
 \ge C_{\alpha,s}j\longrightarrow\infty.
 \label{eq:block-consequences}
\end{equation}
Thus the same amplitudes that ensure convergence leave a residual lower
bound that the growing gain amplifies without bound along one subsequence.

It remains to meet the prescribed smoothness bound. Within each block,
the interpolated derivative is constant; only the joins between gradient
levels require control. These joins occur across the space between
consecutive query points. The velocity built up during the preceding
block provides enough separation to keep their slopes below
$L$ once the first scale is sufficiently large. The explicit choice is
given in \cref{alg:inverse-construction}, and the estimates are proved
in Appendix~\ref{app:no-gain-proof}. A concrete instance with these plateaus
and continuous affine joins is shown in Appendix~\ref{app:construction-instances}.

All blocks are chosen for the given $G$ and incorporated into a single
derivative, with their query points accumulating at zero. The same
fixed loss therefore realizes every lower bound in
\eqref{eq:block-consequences}.

\subsection{Theorem 2: matching the orbit to power geometry}\label{sec:power}

We now use the inverse construction of \cref{sec:inverse-framework}
to obtain a loss with local $p$-power geometry and an exact power-rate
tail. Following the gradient-design approach of \cref{sec:blocks},
we prescribe a power-law sequence whose exponent links the shape of
the loss to the decay of its orbit. \Cref{fig:power-construction}
summarizes the construction.

\begin{figure}[!htb]
\centering
\includegraphics[width=\linewidth]{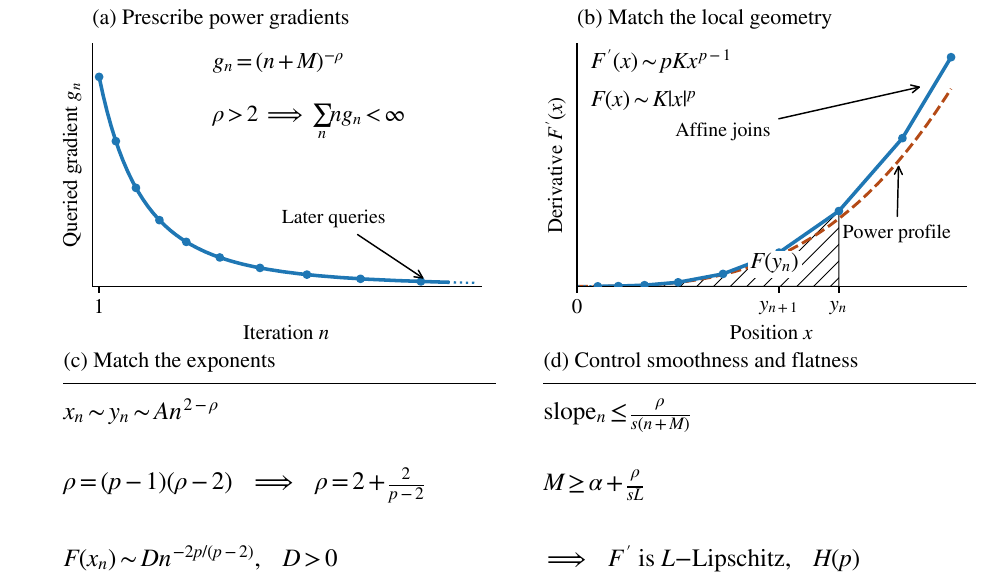}
\caption{Power construction; axes are not to scale. Blue shows the gradients
(a) and affine derivative (b); orange marks the power profile, and hatching
represents $F(y_n)$. Exponent matching yields the rate (c); the shift
ensures smoothness and flatness (d). Here $p>2$ and $A,K,D>0$.}
\label{fig:power-construction}
\end{figure}

For the prescribed $p>2$, choose
\begin{equation}
 g_n=(n+M)^{-\rho},\qquad
 \rho=2+\frac2{p-2},\qquad M>0.
 \label{eq:power-schedule}
\end{equation}
The inverse construction yields
\[
 x_n\sim y_n\sim An^{-2/(p-2)},\qquad A>0.
\]
With this choice of $\rho$, the gradients $g_n=F'(y_n)$ are
asymptotically proportional to $y_n^{p-1}$, the derivative behavior
of a $p$-power loss. This is the link between the temporal gradient
sequence and the spatial geometry we want to construct.

Interpolating the derivative and integrating it as in
\cref{sec:inverse-framework} gives
\[
 F(x)\sim K|x|^p\quad(x\to0),\qquad
 F(x_n)\sim Dn^{-2p/(p-2)}\quad(n\to\infty),\qquad K,D>0.
\]
Thus the same fixed loss has local $p$-power growth and attains the
upper rate in \cref{thm:power} with a positive asymptotic constant.
This establishes sharpness at the stated exponent.

Choosing the fixed shift $M$ sufficiently large ensures global
smoothness and the required flatness without changing the asymptotic
power. Appendix~\ref{app:power} gives the parameter checks, interpolation
estimates, and extension to any fixed finite dimension, together with
a concrete $p=4$ illustration. \Cref{alg:inverse-construction}
records the complete choices for both constructions.

\begin{algorithm}[!htbp]
\caption{Complete fixed-loss construction for Theorems 1 and 2}
\label{alg:inverse-construction}
\begin{algorithmic}[1]
\Require Parameters of Theorem 1 or 2, with $\alpha>3$, $L>0$,
$0<s<L^{-1}$, dimension $d\ge1$, and initial distance $R>0$ for Theorem 1.
\State For Theorem 1, set $N_0=0$ and choose integers $N_j>2N_{j-1}$ with
\Statex \hspace{\algorithmicindent}
$N_j\ge\max\{8,3(\alpha+1)e^\alpha/(sL)\}$ and $G(N_j)\ge j4^j$.
\Statex \hspace{\algorithmicindent}
Set $g_k=2^{-j}/N_j^2$ for $2N_{j-1}<k\le2N_j$.
\State For Theorem 2, set $\rho=2+2/(p-2)$, choose $M\ge\alpha+\rho/(sL)$,
and set $g_n=(n+M)^{-\rho}$.
\State Set $v_1=0$ and $v_{n+1}=\beta_nv_n+sg_n$ for every $n\ge1$.
\State Set $x_n=\sum_{j=n+1}^{\infty}v_j$, $x_0=x_1$, and
$y_n=x_n-\beta_nv_n$.
\State Set $\psi(0)=0$, $\psi(y_n)=g_n$; interpolate affinely between
consecutive nodes.
\State Set $\psi(x)=g_1$ for $x\ge y_1$ and
$F(x)=\int_0^{|x|}\psi(u)\,du$.
\State For Theorem 1, set $h(t)=Lt^2/2$ and $\delta=R/x_1$;
replace $(F,x_n,y_n)$ by
\Statex \hspace{\algorithmicindent}
$(\delta^2F(\,\cdot\,/\delta),\delta x_n,\delta y_n)$.
\State For Theorem 2, set $h(0)=0$ and
$h'(t)=\frac{L}{p-1}\operatorname{sgn}(t)\min\{|t|^{p-1},1\}$.
\State Define $\widehat F(z)=F(z_1)+\sum_{i=2}^dh(z_i)$.
\State \Return $\widehat F$ and its exact CD orbit
 $\widehat x_n=(x_n,0,\ldots,0)$, $\widehat y_n=(y_n,0,\ldots,0)$.
\end{algorithmic}
\end{algorithm}

\section{Numerical experiments on the fixed losses}\label{sec:experiments}

We run CD on the fixed losses with $L=1$, $s=1/2$, and $x_0=x_1=1$.
\Cref{fig:main-experiments} summarizes the results;
Appendix~\ref{app:numerical-witnesses} gives the protocols and numerical checks.

\begin{figure}[!htbp]
\centering
\includegraphics[width=\linewidth]{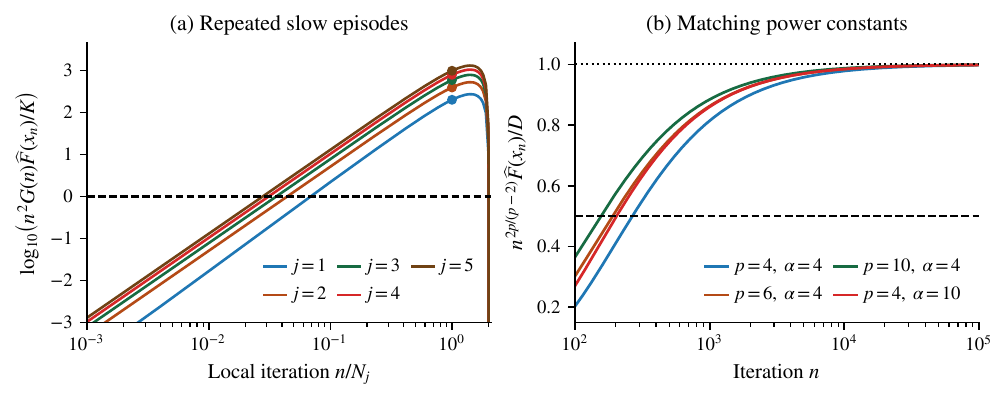}
\caption{CD on the fixed witnesses. Left: five windows of one orbit;
crossing zero means exceeding $K/[n^2G(n)]$, and dots mark $N_j$.
Right: power witnesses approach $1$ (dotted), exceeding $1/2$ (dashed).
$\widehat F$ omits a bounded integration tail; $K$ and $D$ are fixed by construction.}
\label{fig:main-experiments}
\end{figure}

\subsection{Repeated exceedances for a double-logarithmic gain}

For \cref{thm:no-gain}, take $\alpha=4$ and
$G(n)=\log\log(e^e+n)$. With $a_j=(1-r)r^{j-1}$, $r=0.999$,
the scales in Appendix~\ref{app:numerical-witnesses} give
$N_j^2G(N_j)F(x_{N_j})/K\ge2j$, with a fixed $K>0$.
The first five computed ratios are approximately $200$, $394$, $590$,
$786$, and $983$. Local time $n/N_j$ makes each crossing visible.
Because $N_5=2^{34437}$, constant-gradient intervals are traversed by
the closed-form discrete recurrence, with individual oracle steps near
joins. Direct-step and higher-precision checks validate this batching.
The five displayed episodes illustrate the construction; infinitely many
exceedances follow from the bound $2j$.

\subsection{Attaining the power-rate constants}

For \cref{thm:power}, take $g_n=(n+32)^{-\rho}$,
$\rho=2+2/(p-2)$, with $(p,\alpha)=(4,4),(6,4),(10,4),(4,10)$.
Each run executes $10^5$ iterations. Using the analytic constant $D$
without fitting, the final ratios $n^{2p/(p-2)}\widehat F(x_n)/D$
lie between $0.9978$ and $0.9987$. All four observed tails stay above
$1/2$, consistent with the predicted positive asymptotic constant.

\FloatBarrier

\section{Conclusion and limitation}\label{sec:conclusion}

We prove that the strict $o(n^{-2})$ guarantee for CD has no universal
remainder rate: every divergent gain is defeated along a subsequence by
one fixed smooth convex loss. We also construct fixed losses with
$F(x_n)-F^*\sim Dn^{-2p/(p-2)}$, $D>0$, attaining the established upper
rates under the stated geometry and damping assumptions. Both results
follow from prescribing queried gradients and recovering a loss whose
exact discrete orbit realizes the design.

Our analysis concerns deterministic full-gradient CD on convex
objectives. The no-gain loss depends on $G$ and has infinitely many
curvature transitions near its minimizer; it does not imply that every
loss is slow. Stochastic gradients, adaptive restarts, and nonconvex
objectives require separate analysis.

Two questions remain unresolved. First, what is the weakest damping
that guarantees the exponent $2p/(p-2)$ under local $p$-power growth
alone? Our examples establish sharpness of the exponent in the stated
regimes, but do not establish necessity of the damping threshold.
Second, how can these local asymptotic rates be converted into sharp
finite-time bounds with explicit constants and onset times? This requires
quantifying how local geometry controls the rate and what global
information controls entry into the neighborhood where it applies.

\label{maintext:end}
\section*{AI Use Statement}
The initial research, manuscript draft, and Lean~4 formalization were
produced autonomously by an in-house research agent
system powered by GPT-5.6 Sol and GPT-6 Astra. The human authors defined the research
objectives. They subsequently used GPT-6 Astra to assist in verifying
and correcting the results, refining the arguments, and preparing the
final manuscript.

\bibliography{references}
\bibliographystyle{iclr2027_conference}
\clearpage
\appendix
\raggedbottom
\section{Proof roadmap and standing notation}\label{app:overview}

The proofs have two construction branches and one upper-bound branch.
Both constructions first realize prescribed gradients as an exact CD
orbit of a single smooth convex loss. Constant-gradient blocks then
give the arbitrary-gain obstruction, while a shifted power sequence
gives the sharp geometry-dependent examples. The upper bounds use
established energy estimates, with a localization argument that removes
coercivity in finite dimension. \Cref{fig:proof-dependencies} records
these dependencies. The notation and proof targets are restated below
so that the appendix can be read independently of the main exposition.

\begin{figure}[H]
\centering
\begingroup
\resizebox{\linewidth}{!}{%
\begin{tikzpicture}[
  box/.style={draw=black, line width=.8pt, fill=white, text=black,
    text width=4.15cm, minimum height=.95cm, align=center,
    inner sep=5pt, font=\small},
  arrow/.style={draw=black,-{Latex[length=2mm]},line width=.8pt},
  goal/.style={box,line width=1.1pt}
]
\node[box,text width=7.2cm] (inverse) at (2.5,0)
  {Lemma~\ref{lem:inverse}: inverse orbit and smooth interpolation};
\node[box] (cd) at (10,0)
  {CD convergence\\\citep{Chambolle2015}};

\node[box] (block) at (0,-1.65)
  {Lemma~\ref{lem:block}\\constant-gradient block bound};
\node[box] (power) at (5,-1.65)
  {Lemma~\ref{lem:power-leading}\\exact power asymptotics};
\node[box] (local) at (10,-1.65)
  {Lemma~\ref{lem:localization}\\eventual local geometry};

\node[goal] (nogain) at (0,-3.3)
  {Theorem~\ref{thm:no-gain}\\choose scales, rescale, embed\\Appendix~\ref{app:no-gain-proof}};
\node[box] (flat) at (5,-3.3)
  {Lemma~\ref{lem:power-flatness}\\endpoint flatness $H(p)$};
\node[box] (upper) at (10,-3.3)
  {Geometry upper bounds\\cited energies + localization\\Appendix~\ref{app:upper-transfer}};

\node[box] (embed) at (5,-4.95)
  {Lemma~\ref{lem:power-embedding}\\fixed-dimensional witnesses};
\node[goal,text width=7.2cm] (sharp) at (7.5,-6.6)
  {Theorem~\ref{thm:power}: matching upper and lower rates\\
   $F(x_n)-F^*\sim Dn^{-2p/(p-2)}$, $D>0$, on the witness};

\draw[arrow] (inverse.south) -- ++(0,-.3) -| (block.north);
\draw[arrow] (inverse.south) -- ++(0,-.3) -| (power.north);
\draw[arrow] (block) -- (nogain);
\draw[arrow] (power) -- (flat);
\draw[arrow] (flat) -- (embed);
\draw[arrow] (cd) -- (local);
\draw[arrow] (local) -- (upper);
\draw[arrow] (embed.south) -- (5,-5.75) -| ([xshift=-1.2cm]sharp.north);
\draw[arrow] (upper.south) -- (10,-5.75) -| ([xshift=1.2cm]sharp.north);
\end{tikzpicture}%
}
\endgroup
\caption{Proof dependencies. An arrow indicates an input to the next
argument. Numbered lemmas are proved in this appendix. CD convergence
is taken from \citet{Chambolle2015}; the upper-rate arguments use
\citet{Apidopoulos2021} and \citet{aujol2024strongconvergencefistaiterates}. The block and power constructions are
independent of those upper-rate arguments.}
\label{fig:proof-dependencies}
\end{figure}
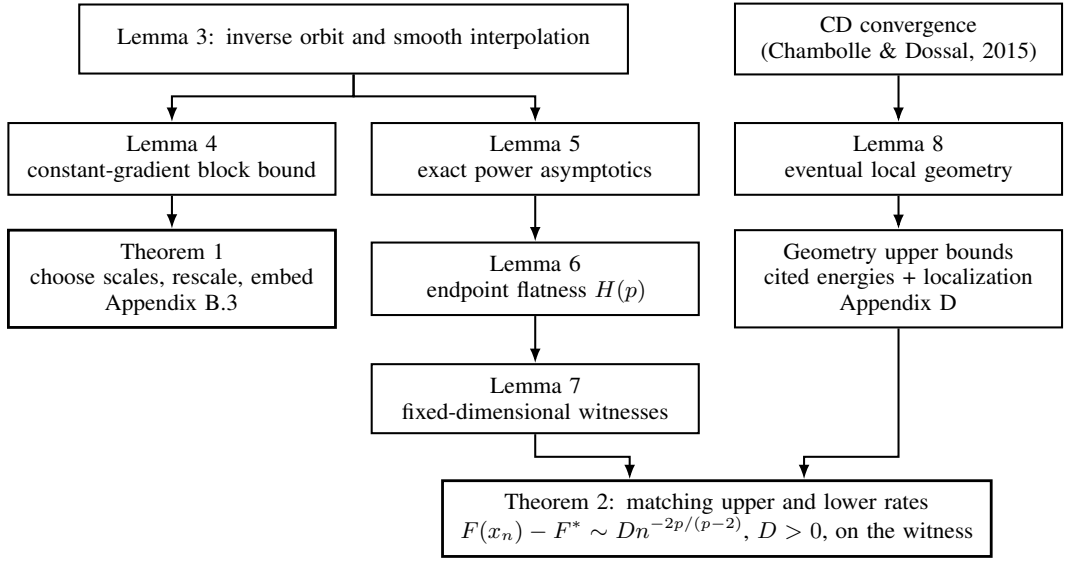

\subsection{What each part of the proof must accomplish}

Appendix~\ref{app:inverse} first recovers positions from the full velocity tail.
The main difficulty is to make all infinitely many queries compatible
with a globally Lipschitz derivative at their accumulation point. The
general interpolation bound is $1/s$, which is larger than the prescribed
$L$ when $s<L^{-1}$. The block construction therefore needs a separate
curvature estimate. Its scale choices balance weighted summability
against the lower bound at infinitely many selected indices.

Appendix~\ref{app:power} replaces blocks by a shifted power sequence. Leading
asymptotics determine both the residual exponent and local growth, but
do not settle the endpoint inequality $pF(x)\le xF'(x)$. The decisive
step is to retain the first correction, choose its sign through the
shift, and show that affine interpolation introduces a smaller error.
A separable extension then preserves this geometry in every fixed
finite dimension.

Appendix~\ref{app:upper-transfer} proves the upper statements. It identifies the
published iteration convention and explains exactly where local
geometry becomes valid. For growth alone, the technical difficulty is
that the Lyapunov energy has a negative distance term; the growth
inequality controls that term and closes the residual estimate.
Appendix~\ref{app:quadratic} treats the quadratic boundary separately, and
Appendix~\ref{app:supplement} introduces the supplementary materials and
records the scope of the Lean~4 formalization.

\subsection{Setting, recurrence, and local geometry}\label{app:setting}

Let $F:\R^d\to\R$ be convex, attain its minimum $F^*$, and have a
globally $L$-Lipschitz gradient, where $d\ge1$ is fixed and $L>0$.
Norms and inner products are Euclidean. We write $C^{1,1}$ for a
differentiable loss with globally Lipschitz gradient, and call $F$
coercive when $F(x)\to\infty$ as $\|x\|\to\infty$.
For $\alpha>3$ and $0<s<L^{-1}$, the Chambolle--Dossal (CD) recurrence is
\begin{equation}
 \beta_n=\frac{n-1}{n+\alpha-1},\qquad
 y_n=x_n+\beta_n(x_n-x_{n-1}),\qquad
 x_{n+1}=y_n-s\nabla F(y_n),\quad n\ge1.
 \label{eq:app-cd}
\end{equation}
The queried gradients are $g_n=\nabla F(y_n)$; the residual and its
normalization are $r_n=F(x_n)-F^*$ and $q_n=n^2r_n$.
Since $\beta_1=0$, $x_0$ does not affect later iterates. We use
$x_0=x_1$ and normalize our constructed losses so that $F^*=F(0)=0$.

For a unique minimizer $x^*$, the local conditions used here are
\begin{align}
 H(\beta):\quad F(x)-F^*&\le
 \frac1\beta\langle\nabla F(x),x-x^*\rangle
 &&(\|x-x^*\|\le\eta),\label{eq:app-H}\\
 L(p):\quad F(x)-F^*&\ge K\|x-x^*\|^p
 &&(\|x-x^*\|\le\varepsilon),\label{eq:app-Lp}
\end{align}
for some $\eta,\varepsilon,K>0$. The flatness exponent $\beta>2$ is
distinct from the momentum coefficient $\beta_n$. If the minimizer
set $X^*$ need not be a singleton, we use
\begin{equation}
 L_{\mathrm{set}}(p):\quad F(x)-F^*\ge K\dist(x,X^*)^p
 \quad\text{whenever }\dist(x,X^*)\le\varepsilon.
 \label{eq:app-Lset}
\end{equation}
The same constants apply throughout the indicated neighborhood.
Here $p>2$, and $\dist(x,X^*)$ is the Euclidean distance to the
minimizer set. For a unique minimizer, $L_{\mathrm{set}}(p)$ is $L(p)$.
The symbols $O$, $o$, and $\sim$ refer to the stated limit; in
particular, $a_n\sim b_n$ means $a_n/b_n\to1$.

\subsection{The conclusions to be proved}\label{app:targets}

For \cref{thm:no-gain}, fix any positive nondecreasing gain
$G(n)\to\infty$ and initial distance $R>0$, in addition to
$\alpha,L,s$. We construct one even, coercive scalar convex loss with
$L$-Lipschitz derivative, unique minimizer zero, and an exact CD orbit
from $x_0=x_1=R$ such that
\[
 \sup_{n\ge1}n^2G(n)F(x_n)=\infty.
\]
The same residual sequence must be realizable in every fixed finite
dimension. This unboundedness conclusion coexists with the established
$r_n=o(n^{-2})$ guarantee \citep{AP16}; the gain-weighted residual
is unbounded on a subsequence, not necessarily divergent at every index.

For \cref{thm:power}, put $q=2p/(p-2)$. We establish the upper bound
$r_n=O(n^{-q})$ for every initialization under either of these regimes:
\begin{enumerate}[label=(\roman*),leftmargin=*]
 \item unique minimizer, $H(\beta)$ and $L(p)$, with
 $2<\beta\le p$ and $\alpha\ge(\beta+2)/(\beta-2)$;
 \item $L_{\mathrm{set}}(p)$ and $\alpha>5+8/(p-2)$.
\end{enumerate}
For every admissible parameter choice and fixed finite dimension, we
also construct a coercive loss with unique minimizer and some
initialization for which $r_n\sim Dn^{-q}$, $D>0$. These examples
satisfy $H(p)$ as well as both growth formulations. Their positive
asymptotic constant rules out any common divergent multiplicative
improvement of the upper rate.

\section{Inverse construction and the no-gain theorem}\label{app:inverse}

The task is to build one loss that defeats a prescribed gain at infinitely
many indices. We first give a general procedure for realizing a positive
gradient sequence as an exact orbit. A block estimate then quantifies how
much objective error a constant-gradient interval retains. Finally, we
choose all the blocks at once, verifying that their accumulation at zero
preserves the required smoothness.

Prescribing an orbit and recovering a scalar objective has a continuous
precedent in \citet[Lemma~2.7]{siegel2023qualitativedifferencegradientflows}. Interpolation of prescribed
first-order data is also a standard tool in smooth convex optimization
\citep{Taylor2017}. Here the query locations are determined by
the discrete momentum recurrence, and the interpolation must work at
countably many nodes accumulating at the minimizer. The next lemma
supplies this compatibility directly.

\subsection{From a gradient sequence to one smooth convex loss}

The velocity $v_n=x_{n-1}-x_n$ measures motion toward zero. Its recurrence
can be solved forward from the gradients, but the positions are defined
by summing the entire remaining velocity tail. This is what fixes the
limiting point at zero without knowing the loss in advance.

\begin{lemma}[Inverse orbit and smooth interpolation]\label{lem:inverse}
Let $\alpha>3$ and $s>0$. Let $(g_n)_{n\ge1}$ be positive,
nonincreasing, and satisfy $\sum_{n\ge1}ng_n<\infty$. With
$\beta_n=(n-1)/(n+\alpha-1)$, define
\begin{equation}
 \begin{aligned}
 v_1&=0,& v_{n+1}&=\beta_n v_n+sg_n,\\
 x_n&=\sum_{j=n+1}^{\infty}v_j,& y_n&=x_n-\beta_n v_n,
 \end{aligned}
 \qquad n\ge1,\qquad x_0=x_1.
 \label{eq:app-inverse}
\end{equation}
These sums are finite. In particular, $\sum_n v_n<\infty$,
$x_n\downarrow0$, $y_n\downarrow0$, and $y_n-sg_n=x_{n+1}$.
There is a convex $C^{1,1}$ function $F:\R\to\R$ with
$F'(y_n)=g_n$ for all $n$, so these points are an exact CD orbit.
The function may be chosen even, coercive, and uniquely minimized at zero.
Its derivative is $1/s$-Lipschitz; if all interpolation slopes are at most
$L$, it is $L$-Lipschitz.
\end{lemma}

\begin{proof}[Proof of \cref{lem:inverse}]
\proofstep{1}{Summability of the recovered velocities}
For integers $1\le k\le n$, with the empty-product convention
$P_{k,k}=1$, the products of momentum coefficients satisfy
\begin{equation}
 P_{k,n}:=\prod_{i=k+1}^{n}\beta_i
 =\frac{\Gamma(n)\Gamma(k+\alpha)}{\Gamma(k)\Gamma(n+\alpha)},
 \qquad
 \sum_{n=k}^{\infty}P_{k,n}=\frac{k+\alpha-1}{\alpha-1}.
 \label{eq:products}
\end{equation}
Here $\Gamma$ is the gamma function. Its recurrence
$\Gamma(z+1)=z\Gamma(z)$ for $z>0$ gives
\[
 \prod_{i=k+1}^{n}\beta_i
 =\frac{k(k+1)\cdots(n-1)}{(k+\alpha)(k+\alpha+1)\cdots(n+\alpha-1)}
 =\frac{\Gamma(n)/\Gamma(k)}{\Gamma(n+\alpha)/\Gamma(k+\alpha)}.
\]
To compute the infinite sum, use the classical beta integral
\citep[Eq.~5.12.1]{NIST:DLMF}:
\[
 \frac{\Gamma(n)}{\Gamma(n+\alpha)}
 =\frac1{\Gamma(\alpha)}
   \int_0^1 t^{n-1}(1-t)^{\alpha-1}\,dt.
\]
All integrands are nonnegative, so Tonelli's theorem permits exchanging
the sum and the integral \citep[Theorem~2.37]{Folland99}. The geometric
series $\sum_{n=k}^{\infty}t^{n-1}=t^{k-1}/(1-t)$ for $0<t<1$ then yields
\begin{align*}
 \sum_{n=k}^{\infty}P_{k,n}
 &=\frac{\Gamma(k+\alpha)}{\Gamma(k)\Gamma(\alpha)}
   \int_0^1\left(\sum_{n=k}^{\infty}t^{n-1}\right)(1-t)^{\alpha-1}\,dt\\
 &=\frac{\Gamma(k+\alpha)}{\Gamma(k)\Gamma(\alpha)}
   \int_0^1 t^{k-1}(1-t)^{\alpha-2}\,dt\\
 &=\frac{\Gamma(k+\alpha)}{\Gamma(k)\Gamma(\alpha)}
   \frac{\Gamma(k)\Gamma(\alpha-1)}{\Gamma(k+\alpha-1)}
 =\frac{k+\alpha-1}{\alpha-1}.
\end{align*}
The second beta integral is finite because $k\ge1$ and $\alpha>1$,
as ensured by our assumption $\alpha>3$. The last equality uses
$\Gamma(k+\alpha)=(k+\alpha-1)\Gamma(k+\alpha-1)$ and
$\Gamma(\alpha)=(\alpha-1)\Gamma(\alpha-1)$.
We next expand the velocity recurrence. The base case is
$v_2=sg_1=sg_1P_{1,1}$. If the expansion holds at index $n$, then
\[
 v_{n+2}
 =\beta_{n+1}s\sum_{k=1}^n g_kP_{k,n}+sg_{n+1}
 =s\sum_{k=1}^{n+1}g_kP_{k,n+1},
\]
because $\beta_{n+1}P_{k,n}=P_{k,n+1}$ and
$P_{n+1,n+1}=1$. Thus induction and the product sum give
\begin{equation}
 v_{n+1}=s\sum_{k=1}^n g_kP_{k,n},\qquad
 \sum_{n\ge1}v_{n+1}=\frac{s}{\alpha-1}
       \sum_{k\ge1}(k+\alpha-1)g_k<\infty.
 \label{eq:vproduct}
\end{equation}
For completeness, the sum exchange and the finiteness assertion in
\eqref{eq:vproduct} are
\begin{align*}
 \sum_{n=1}^{\infty}v_{n+1}
 &=s\sum_{n=1}^{\infty}\sum_{k=1}^n g_kP_{k,n}
 =s\sum_{k=1}^{\infty}g_k\sum_{n=k}^{\infty}P_{k,n}\\
 &=\frac{s}{\alpha-1}\sum_{k=1}^{\infty}(k+\alpha-1)g_k
 \le\frac{s\alpha}{\alpha-1}\sum_{k=1}^{\infty}kg_k<\infty.
\end{align*}
Tonelli applies because every summand is nonnegative; the inequality uses
$k+\alpha-1\le\alpha k$ for $k\ge1$ and $\alpha>1$.
Also $0\le\beta_n<1$ and $g_n>0$ imply $v_{n+1}\ge sg_n>0$.
Consequently, the tails defining $x_n$ are finite,
$x_n-x_{n+1}=v_{n+1}>0$, and $x_n\downarrow0$.

To express these tails directly in terms of the gradients, the same
beta-integral calculation, now starting at $m=n$, gives for $1\le k\le n$
\begin{align*}
 \sum_{m=n}^{\infty}P_{k,m}
 &=\frac{\Gamma(k+\alpha)}{\Gamma(k)\Gamma(\alpha)}
   \int_0^1 t^{n-1}(1-t)^{\alpha-2}\,dt\\
 &=\frac{\Gamma(k+\alpha)\Gamma(n)}
 {\Gamma(k)(\alpha-1)\Gamma(n+\alpha-1)}.
\end{align*}
Substituting the velocity expansion into $x_n$ and again exchanging
nonnegative sums yields
\begin{align*}
 x_n&=\sum_{m=n}^{\infty}v_{m+1}
 =s\sum_{k=1}^{\infty}g_k\sum_{m=\max\{n,k\}}^{\infty}P_{k,m}\\
 &=s\sum_{k=1}^n g_k\sum_{m=n}^{\infty}P_{k,m}
   +s\sum_{k=n+1}^{\infty}g_k\sum_{m=k}^{\infty}P_{k,m}.
\end{align*}
Use the tail identity in the first term and \eqref{eq:products} in the
second to obtain
\begin{align}
 x_n=\frac{s}{\alpha-1}\left[
 \frac{\Gamma(n)}{\Gamma(n+\alpha-1)}
 \sum_{k=1}^{n}\frac{\Gamma(k+\alpha)}{\Gamma(k)}g_k
 +\sum_{k=n+1}^{\infty}(k+\alpha-1)g_k\right].
 \label{eq:position-transform}
\end{align}
\proofstep{2}{Ordering the query points}
The tail differences give $x_{n-1}-x_n=v_n$ for $n\ge2$;
the same identity holds at $n=1$ because $x_0=x_1$ and $v_1=0$.
Using $\beta_nv_n=v_{n+1}-sg_n$, we therefore have
\[
 y_n=x_n-\beta_nv_n
     =x_n-v_{n+1}+sg_n=x_{n+1}+sg_n>0.
\]
Subtracting this identity at consecutive indices gives
\begin{equation}
 y_n-y_{n+1}=v_{n+2}+s(g_n-g_{n+1})>0.
 \label{eq:mesh}
\end{equation}
The right side is strictly positive because $v_{n+2}>0$ and
$g_n\ge g_{n+1}$. Moreover, $\sum_n ng_n<\infty$ implies
$ng_n\to0$, hence $g_n\to0$. Together with $x_{n+1}\to0$, the
identity $y_n=x_{n+1}+sg_n$ proves $y_n\downarrow0$.

\proofstep{3}{Interpolating through the accumulation point}
Set $\psi(0)=0$ and $\psi(y_n)=g_n$, interpolate affinely on each
$[y_{n+1},y_n]$, and set $\psi(x)=g_1$ for $x\ge y_1$.
Explicitly, on an interpolation interval,
\[
 \psi(x)=g_{n+1}
  +\frac{g_n-g_{n+1}}{y_n-y_{n+1}}(x-y_{n+1}).
\]
The derivative so defined is continuous, positive on $(0,\infty)$,
and nondecreasing. Its affine slopes satisfy
\begin{equation}
 0\le\ell_n=
 \frac{g_n-g_{n+1}}{v_{n+2}+s(g_n-g_{n+1})}\le\frac1s.
 \label{eq:slope}
\end{equation}
Here the denominator is positive. If $g_n=g_{n+1}$ the slope is zero;
otherwise the denominator is at least $s(g_n-g_{n+1})$, giving the
upper bound. For $0\le x\le y_n$, monotonicity gives
$0\le\psi(x)\le g_n\to0$, proving continuity at zero.

Let $\Lambda=1/s$, or $\Lambda=L$ when all $\ell_n\le L$.
Given $0<u<v$, only finitely many nodes lie in $[u,v]$, because
$y_n\to0$. If $u=t_0<\cdots<t_r=v$ partitions this interval at those
nodes, the slope bounds and the constant outer extension give
\[
 0\le\psi(v)-\psi(u)
 =\sum_{i=1}^r\bigl(\psi(t_i)-\psi(t_{i-1})\bigr)
 \le\Lambda\sum_{i=1}^r(t_i-t_{i-1})=\Lambda(v-u).
\]
Letting $u\downarrow0$ also gives $\psi(v)\le\Lambda v$.
Extend $\psi$ oddly to $\R$. The same-sign Lipschitz bounds follow
by reflection; for $u<0<v$,
\[
 |\psi(v)-\psi(u)|=\psi(v)+\psi(-u)
 \le\Lambda(v-u).
\]
The extension is continuous and nondecreasing, so its integral
\[
 F(x)=\int_0^x\psi(u)\,du
     =\int_0^{|x|}\psi(u)\,du
\]
is even and convex, with $F'=\psi$ globally $\Lambda$-Lipschitz.
Thus $F\in C^{1,1}$, including at zero. Since $\psi(u)>0$ for $u>0$,
$F(x)>F(0)=0$ whenever $x\ne0$. Finally, for $|x|\ge y_1$,
\[
 F(x)=F(y_1)+g_1(|x|-y_1)\longrightarrow\infty
 \quad\text{as }|x|\to\infty,
\]
which proves coercivity. The identities $F'(y_n)=g_n$,
$y_n=x_n+\beta_n(x_n-x_{n-1})$, and
$y_n-sg_n=x_{n+1}$ verify both updates of the exact CD recurrence.
\end{proof}

The automatic bound $1/s$ does not yet meet the standing requirement
$L<1/s$. Each construction below must therefore control its own
interpolation slopes. This is the only extra smoothness check needed
once the gradient sequence has been chosen.

\subsection{Why a constant-gradient block retains error}

Consider a block $n\le k\le2n$ on which $g_k=c>0$.
Each update adds $sc$ to the velocity. The momentum products within this
block have a positive lower bound depending only on $\alpha$, so these
contributions accumulate: during the latter part of the block, the
velocity is at least a constant times $scn$. Summing over a number of
steps proportional to $n$ gives a spatial segment of length at least a
constant times $scn^2$. The interpolated derivative equals $c$ on this
segment, so its integral contributes at least $C_{\alpha,s}c^2n^2$ to
$F(x_n)$, for some $C_{\alpha,s}>0$. Earlier gradients contribute only
nonnegative terms to the velocity and can be discarded in this estimate.
The lower bound therefore depends only on the current block.

\begin{lemma}[Constant-gradient block]\label{lem:block}
If $g_k=c$ for $n\le k\le2n$ and $n\ge8$, then the objective from \cref{lem:inverse} satisfies
\[
 F(x_n)\ge \frac{s e^{-\alpha}}6c^2n^2.
\]
\end{lemma}

\begin{proof}[Proof of \cref{lem:block}]
Since $0<y_n\le x_n$ and $F$ is nondecreasing on the positive axis,
$F(x_n)\ge F(y_n)$. On $[y_{k+1},y_k]$, the derivative is at least
$g_{k+1}$. Partitioning $[0,y_n]$ into these intervals therefore gives
\[
 F(x_n)\ge F(y_n)
 \ge\sum_{k=n}^{\infty}g_{k+1}(y_k-y_{k+1})
 \ge\sum_{k=n}^{\infty}g_{k+1}v_{k+2}.
\]
The first sum follows by integrating interval by interval and taking
the limit at zero; the second uses \eqref{eq:mesh} and
$g_k-g_{k+1}\ge0$.
For $n\le j\le k+1\le2n$, apply
$\log(1-u)\ge-u/(1-u)$ for $0\le u<1$ to the product in
\eqref{eq:products}:
\[
 \log P_{j,k+1}
 =\sum_{i=j+1}^{k+1}\log\left(1-\frac{\alpha}{i+\alpha-1}\right)
 \ge-\alpha\sum_{i=j+1}^{k+1}\frac1{i-1}\ge-\alpha.
\]
The logarithmic inequality follows from
$-\log(1-u)=\int_0^u(1-t)^{-1}\,dt\le u/(1-u)$.
Here $i\ge n+1\ge9$ makes $u=\alpha/(i+\alpha-1)$ lie in $(0,1)$,
and
\[
 \sum_{i=j+1}^{k+1}\frac1{i-1}
 \le\frac{k+1-j}{j}\le\frac{2n-n}{n}=1.
\]
Thus $P_{j,k+1}\ge e^{-\alpha}$ throughout this range.
For each $\lceil3n/2\rceil\le k\le2n-1$, retain only the block
terms $n\le j\le k+1$ in the velocity expansion:
\[
 v_{k+2}=s\sum_{j=1}^{k+1}g_jP_{j,k+1}
 \ge sc\sum_{j=n}^{k+1}P_{j,k+1}
 \ge sc e^{-\alpha}(k-n+2)
 \ge\frac12sc e^{-\alpha}n.
\]
There are $2n-\lceil3n/2\rceil=\lfloor n/2\rfloor\ge n/3$
such indices when $n\ge8$, and $g_{k+1}=c$ at all of them. Hence
\[
 F(x_n)\ge\sum_{k=\lceil3n/2\rceil}^{2n-1}c v_{k+2}
 \ge\frac n3\,c\,\frac12sc e^{-\alpha}n
 =\frac{se^{-\alpha}}6c^2n^2.
\]
\end{proof}

\subsection{Choosing all scales and completing the no-gain proof}
\label{app:no-gain-proof}

The choice of scales must satisfy two competing demands. The weighted
gradient sum must remain finite so that the orbit converges to zero,
while the gain-weighted normalized residual must become arbitrarily
large along a subsequence. We first choose a summable sequence $(a_j)$
to control each block's contribution to $\sum_k kg_k$, then place each
block late enough that $G$ is large. This choice of separated scales
is analogous to the scalar
flow construction in \citet[Example~2.9]{siegel2023qualitativedifferencegradientflows}; the interpolation-slope
estimate below is what makes it compatible with the fixed discrete step.

\begin{proof}[Proof of \cref{thm:no-gain}]
\proofstep{1}{Choose one infinite gradient sequence}
Work first in dimension one. Let $a_j=2^{-j}$ and $N_0=0$.
Recursively choose integers $N_j$ so that
\begin{equation}
 N_j>2N_{j-1},\qquad N_j\ge8,\qquad G(N_j)\ge j4^j,
 \qquad
 N_1\ge\frac{3(\alpha+1)e^{\alpha}}{sL}.
 \label{eq:first-scale-slack}
\end{equation}
Such choices exist because $G(n)\to\infty$. Set
$c_j=a_j/N_j^2$ and $g_k=c_j$ for $2N_{j-1}<k\le2N_j$.
Indeed, after $N_{j-1}$ is fixed, the gain inequality holds at every
sufficiently large integer, so the separation and size conditions can
be imposed simultaneously. Since $N_j\to\infty$, the blocks cover
all positive integers without gaps. Their heights satisfy
\[
 \frac{c_{j+1}}{c_j}
 =\frac12\left(\frac{N_j}{N_{j+1}}\right)^2<\frac18,
\]
so the full gradient sequence is positive and nonincreasing. Moreover,
\[
 \sum_{k\ge1}kg_k
 \le\sum_{j\ge1}c_j\sum_{k=1}^{2N_j}k
 \le3\sum_{j\ge1}c_jN_j^2
 =3\sum_{j\ge1}a_j<\infty.
\]
Here $\sum_{k=1}^{2N_j}k=N_j(2N_j+1)\le3N_j^2$ since $N_j\ge1$,
and $\sum_{j\ge1}a_j=1$.
Thus \cref{lem:inverse} produces a single even, coercive loss with
unique minimizer zero and an exact CD orbit.

\proofstep{2}{Enforce the prescribed smoothness at every block transition}
Inside a block, consecutive gradients agree, so the interpolation slope
vanishes. At a transition $n=2N_j$, the last $N_j$ terms in
\eqref{eq:vproduct}, indexed by $N_j+1\le k\le2N_j$, lie inside
block $j$. The product estimate in \cref{lem:block} gives
\[
 v_{n+1}\ge s\sum_{k=N_j+1}^{2N_j}c_jP_{k,2N_j}
 \ge sc_jN_je^{-\alpha}.
\]
At this transition $g_n=c_j$ and $g_{n+1}=c_{j+1}$. Since
$y_n=x_{n+1}+sg_n$, the spacing between the interpolation nodes is
\begin{align*}
 y_n-y_{n+1}
 &=x_{n+1}-x_{n+2}+s(g_n-g_{n+1})\\
 &=v_{n+2}+s(g_n-g_{n+1})\\
 &=\beta_{n+1}v_{n+1}+sc_{j+1}+s(c_j-c_{j+1})\\
 &=\beta_{n+1}v_{n+1}+sc_j
 \ge sc_j\left(1+\frac{N_je^{-\alpha}}{\alpha+1}\right).
\end{align*}
The last inequality uses $\beta_{n+1}=n/(n+\alpha)\ge1/(\alpha+1)$,
equivalent to $\alpha(n-1)\ge0$. The interpolation slope on
$[y_{n+1},y_n]$ is the gradient difference divided by this spacing.
Since $0<c_j-c_{j+1}\le c_j$,
\begin{align*}
 \ell_{2N_j}
 &=\frac{g_n-g_{n+1}}{y_n-y_{n+1}}
 =\frac{c_j-c_{j+1}}{y_n-y_{n+1}}\\
 &\le\frac{1}{s(1+N_je^{-\alpha}/(\alpha+1))}
 \le\frac L3.
\end{align*}
For the final inequality, put
$t_j=N_je^{-\alpha}/(\alpha+1)\ge3/(sL)$ and use
$1/[s(1+t_j)]\le1/(st_j)\le L/3$.
The lower bound on $N_1$ therefore controls every transition, including
those arbitrarily close to zero. The loss has an $L$-Lipschitz derivative.

\proofstep{3}{Violate the gain along infinitely many iterates}
The interval $[N_j,2N_j]$ lies inside block $j$.
Applying \cref{lem:block} at $n=N_j$ and substituting
$c_j=a_j/N_j^2$ gives
\[
 N_j^2F(x_{N_j})\ge\frac{se^{-\alpha}}6a_j^2,
 \qquad
 N_j^2G(N_j)F(x_{N_j})\ge\frac{se^{-\alpha}}6j\longrightarrow\infty.
\]
Indeed, $N_j^2(c_j^2N_j^2)=a_j^2=4^{-j}$ and
$G(N_j)4^{-j}\ge j$.
The objective remains fixed as $j$ grows. Since every finite prefix has
finite residuals, this also rules out any eventual bound with an
instance-dependent constant and starting index.

\proofstep{4}{Prescribe the initial distance and ambient dimension}
The recovered initial point $x_1$ is positive and finite. For any
prescribed $R>0$, put $\delta=R/x_1$ and
$F_\delta(t)=\delta^2F(t/\delta)$. Its derivative is
$F_\delta'(t)=\delta F'(t/\delta)$, with the same Lipschitz constant.
More explicitly,
\[
 |F_\delta'(u)-F_\delta'(v)|
 \le\delta L|u/\delta-v/\delta|=L|u-v|,
 \qquad
 \delta y_n-sF_\delta'(\delta y_n)
 =\delta(y_n-sF'(y_n))=\delta x_{n+1}.
\]
The extrapolation identity also scales by $\delta$, so the scaled
orbit starts at $\delta x_0=\delta x_1=R$ and follows the same CD
recurrence. Since $F_\delta(\delta x_n)=\delta^2F(x_n)$, its
gain-weighted residual remains unbounded.

For any prescribed $d>1$, set
\[
 \widetilde F(z)=F_\delta(z_1)+\frac L2\sum_{i=2}^d z_i^2
 \quad\text{and initialize at }(R,0,\ldots,0).
\]
This loss is coercive, has unique minimizer zero, and has an
$L$-Lipschitz gradient in the Euclidean norm: for $z,w\in\R^d$,
\[
 \|\nabla\widetilde F(z)-\nabla\widetilde F(w)\|^2
 =|F_\delta'(z_1)-F_\delta'(w_1)|^2
   +L^2\sum_{i=2}^d|z_i-w_i|^2
 \le L^2\|z-w\|^2.
\]
To check coercivity, a sublevel $\widetilde F(z)\le H$ with $H\ge0$
forces $F_\delta(z_1)\le H$ and
$\sum_{i=2}^d z_i^2\le2H/L$. The scalar coercivity bounds $z_1$,
so every such sublevel is bounded. Continuity then gives coercivity
in finite dimension. Each summand vanishes only at zero, proving
uniqueness of the minimizer.
If two successive transverse coordinates equal zero, their extrapolated
coordinate and gradient also equal zero; induction therefore keeps
them zero. The first coordinate follows the scaled scalar orbit, and
$\widetilde F(\delta x_n,0,\ldots,0)=\delta^2F(x_n)$.
Both the initial distance and the scalar residual sequence are
preserved. This completes the theorem.
\end{proof}

\section{Exact power witnesses and endpoint flatness}\label{app:power}

We next build the sharp instances for the geometry-dependent rates.
Unlike the block construction, a power gradient sequence has a regular
tail, which makes its velocity, position, and loss asymptotics explicit.
The main additional work is at the endpoint $H(p)$: a limit of
$xF'(x)/F(x)$ equal to $p$ does not say on which side of $p$ the ratio
lies. We will choose a positive first correction and show that it survives
interpolation between the query points.

\subsection{Parameter choice and the required exponent}

For the prescribed
$p>2$, set
\[
 \rho=2+\frac2{p-2},\qquad \delta=\rho-2,\qquad
 g_n=(n+M)^{-\rho}.
\]
Choose a fixed positive shift $M$, subject to the bounds below.
In part~(i), the map $t\mapsto1+4/(t-2)$ decreases for $t>2$, so
\[
 \alpha\ge\frac{\beta+2}{\beta-2}
 =1+\frac4{\beta-2}\ge1+\frac4{p-2}=2\rho-3.
\]
Thus $\rho\le(\alpha+3)/2<\alpha$, where the strict inequality uses
the standing assumption $\alpha>3$. In part~(ii),
\[
 \alpha>5+\frac8{p-2}=4\rho-3>\rho,
\]
because $\rho>2$. Consequently, both regimes give
\begin{equation}
 2<\rho<\alpha,\qquad a:=\alpha-\rho>0.
 \label{eq:power-parameter-range}
\end{equation}
The sequence $(g_n)$ is positive and nonincreasing, and
$\sum_n ng_n\le\sum_n n^{1-\rho}<\infty$, since $M>0$ and
$\rho>2$. We use the exact inverse orbit from
\cref{lem:inverse}, with $x_0=x_1$.

\subsection{Leading asymptotics, local growth, and global smoothness}

The first lemma identifies the exponent and the positive asymptotic
constant. It also checks local growth throughout a neighborhood, rather
than only at the query points, and enforces the prescribed smoothness.

\begin{lemma}[Exact power asymptotics]\label{lem:power-leading}
Fix $\alpha>3$, $L>0$, $0<s<L^{-1}$, $2<\rho<\alpha$, and
$M\ge\rho/(sL)$. For $g_n=(n+M)^{-\rho}$, construct $F$ using the
inverse orbit and the piecewise-affine derivative in the proof of
\cref{lem:inverse}, with $F(0)=0$. Set
\[
 \delta=\rho-2,\quad p=\frac{2(\rho-1)}{\rho-2},\quad
 B=\frac{s}{\alpha+1-\rho},\quad A=\frac B\delta,\quad
 D=\frac{B}{2(\rho-1)}.
\]
The resulting loss is even, coercive, convex, and has an $L$-Lipschitz
derivative and unique minimizer zero. Its exact orbit satisfies
\[
 v_{n+1}\sim Bn^{1-\rho},\qquad
 x_n\sim y_n\sim An^{-\delta},\qquad
 F(x_n)\sim Dn^{-2p/(p-2)}.
\]
Moreover, $F(x)\sim(D/A^p)|x|^p$ as $x\to0$, so $L(p)$ holds.
\end{lemma}

\begin{proof}[Proof of \cref{lem:power-leading}]
\proofstep{1}{Recover the velocity and position asymptotics}
The gamma-product formula \eqref{eq:vproduct} gives
\[
 v_{n+1}
 =s\frac{\Gamma(n)}{\Gamma(n+\alpha)}
   \sum_{k=1}^n\frac{\Gamma(k+\alpha)}{\Gamma(k)}(k+M)^{-\rho}.
\]
The standard gamma-ratio asymptotic
\citep[Eq.~5.11.12]{NIST:DLMF} gives separately
\[
 \frac{\Gamma(k+\alpha)}{\Gamma(k)}\sim k^\alpha,
 \qquad (k+M)^{-\rho}\sim k^{-\rho},
 \qquad \frac{\Gamma(n)}{\Gamma(n+\alpha)}\sim n^{-\alpha}.
\]
The summand is therefore asymptotic to $k^a$, with
$a=\alpha-\rho>0$. Integral comparison gives
\[
 \frac{n^{a+1}}{a+1}
 =\int_0^n t^a\,dt
 \le\sum_{k=1}^n k^a
 \le\int_1^{n+1}t^a\,dt
 \sim\frac{n^{a+1}}{a+1}.
\]
For any fixed relative error, the original summand lies between the
corresponding multiples of $k^a$ once $k$ is large. Its finite initial
sum is $o(n^{a+1})$, so the same asymptotic holds for the full sum.
Multiplying by $sn^{-\alpha}$ yields
\[
 v_{n+1}\sim Bn^{1-\rho},\qquad
 B=\frac{s}{\alpha+1-\rho}.
\]
To sum the velocity tail, use $\rho>2$ and
\[
 \int_n^\infty t^{1-\rho}\,dt
 \le\sum_{m=n}^\infty m^{1-\rho}
 \le n^{1-\rho}+\int_n^\infty t^{1-\rho}\,dt,
 \qquad
 \int_n^\infty t^{1-\rho}\,dt=\frac{n^{2-\rho}}{\rho-2}.
\]
The extra term is $o(n^{2-\rho})$. Bounding $v_{m+1}$ between
$(1\pm\varepsilon)Bm^{1-\rho}$ for every sufficiently large $m$
therefore gives
$x_n=\sum_{m=n}^\infty v_{m+1}\sim Bn^{2-\rho}/(\rho-2)$.
Also $x_{n+1}/x_n\to1$ and
$sg_n/x_n=O(n^{-2})\to0$. Thus $y_n=x_{n+1}+sg_n$ yields
\[
 x_n\sim y_n\sim An^{-\delta},\qquad A=\frac B\delta.
\]
\proofstep{2}{Integrate the derivative and pass from nodes to a neighborhood}
Write $\psi=F'$ for the derivative from this construction. It is affine
on each interpolation interval $[y_{k+1},y_k]$, so
\[
 F(y_n)=\sum_{k=n}^{\infty}
         \frac{g_k+g_{k+1}}2(y_k-y_{k+1}).
\]
Indeed, each term is the exact trapezoidal integral over
$[y_{k+1},y_k]$, and these intervals partition $(0,y_n]$.
The finite sum through $k=N$ equals $F(y_n)-F(y_{N+1})$;
letting $N\to\infty$ gives the formula because $y_{N+1}\to0$
and $F$ is continuous with $F(0)=0$.
The mean value theorem gives
$g_k-g_{k+1}=O(k^{-\rho-1})$, so
$s(g_k-g_{k+1})/v_{k+2}=O(k^{-2})\to0$.
Hence \eqref{eq:mesh} implies
\[
 y_k-y_{k+1}=v_{k+2}+s(g_k-g_{k+1})\sim Bk^{1-\rho},
 \qquad \frac{g_k+g_{k+1}}2\sim k^{-\rho}.
\]
The area summand is asymptotic to $Bk^{1-2\rho}$. Applying the same
tail comparison, now with exponent $1-2\rho<-1$, gives
$\sum_{k=n}^\infty k^{1-2\rho}\sim n^{2-2\rho}/(2\rho-2)$, and thus
\begin{equation}
 F(y_n)\sim Dn^{2-2\rho},\qquad
 D=\frac{B}{2(\rho-1)}.
 \label{eq:power-nodal-asymptotic}
\end{equation}
For $n\ge2$, the identities $y_n=x_n-\beta_nv_n\le x_n$ and
$y_{n-1}=x_n+sg_{n-1}\ge x_n$ show that
\[
 F(y_n)\le F(x_n)\le F(y_{n-1}).
\]
Both bounding quantities divided by $Dn^{2-2\rho}$ tend to one,
because $((n-1)/n)^{2-2\rho}\to1$. This proves the same asymptotic
for $F(x_n)$. The exponent conversion is
\[
 p-2=\frac2\delta,\qquad
 \frac{2p}{p-2}=p\delta=2(\rho-1).
\]
In particular, $F(y_n)/y_n^p\to D/A^p$.
For $x\in[y_{n+1},y_n]$, monotonicity gives
\[
 \frac{F(y_{n+1})}{y_n^p}
 \le\frac{F(x)}{x^p}
 \le\frac{F(y_n)}{y_{n+1}^p}.
\]
For example, the left bound equals
$[F(y_{n+1})/y_{n+1}^p](y_{n+1}/y_n)^p\to D/A^p$;
the right bound has the same limit. This proves the asserted
asymptotic for every $x\downarrow0$, not just at the nodes.
Choose $\varepsilon>0$ such that
$F(x)\ge[D/(2A^p)]x^p$ for $0\le x\le\varepsilon$.
By evenness, $F(x)\ge K|x|^p$ for $|x|\le\varepsilon$, with
$K=D/(2A^p)>0$. This is $L(p)$.

Similarly, on the same interval,
\[
 \frac{g_{n+1}}{y_n^{p-1}}
 \le\frac{\psi(x)}{x^{p-1}}
 \le\frac{g_n}{y_{n+1}^{p-1}}.
\]
Since $\delta(p-1)=\rho$ and $y_{n+1}/y_n\to1$, both bounds tend
to $A^{-(p-1)}$. Consequently, as $x\downarrow0$,
\[
 \frac{xF'(x)}{F(x)}\longrightarrow
 \frac{A^{-(p-1)}}{D/A^p}=\frac AD=p,
\]
where $A/D=2(\rho-1)/(\rho-2)=p$.
Since $F$ is even and $F'$ is odd, the same limit holds as $x\uparrow0$.

\proofstep{3}{Bound all interpolation slopes}
Let $\ell_n$ denote the slope of $\psi$ on $[y_{n+1},y_n]$, as
defined in \eqref{eq:slope}.
To enforce the prescribed smoothness, the recurrence gives
$v_{n+2}=\beta_{n+1}v_{n+1}+sg_{n+1}\ge sg_{n+1}$.
Thus the slope denominator obeys
$v_{n+2}+s(g_n-g_{n+1})\ge sg_n>0$, and
\[
 0\le\ell_n
 \le\frac{g_n-g_{n+1}}{sg_n}
 =\frac1s\left[1-\left(\frac{n+M}{n+M+1}\right)^\rho\right]
 \le\frac{\rho}{s(n+M)}.
\]
The last inequality follows, with $t=n+M>0$, from
\[
 1-\left(\frac{t}{t+1}\right)^\rho
 =1-(1+1/t)^{-\rho}
 =\int_0^{1/t}\rho(1+u)^{-\rho-1}\,du\le\frac\rho t.
\]
Consequently, $M\ge\rho/(sL)$ gives
$\ell_n\le\rho/[s(n+M)]\le\rho/(sM)\le L$ on every interpolation
interval. The construction sets $\psi(0)=0$, continues $\psi$ constantly
as $g_1$ for $x\ge y_1$, and extends it oddly to the negative axis.
Step~3 of the proof of \cref{lem:inverse} shows that these extensions
preserve continuity, monotonicity, and the global $L$-Lipschitz bound,
including at zero. It also proves that
$F(x)=\int_0^x\psi(u)\,du$ is even and coercive, has unique minimizer
zero, and realizes the prescribed CD orbit.

\end{proof}

\subsection{The first correction and endpoint flatness}

The previous lemma implies $xF'(x)/F(x)\to p$, which suffices for
$H(\beta)$ when $\beta<p$. The endpoint $H(p)$ needs more: the
ratio must approach $p$ from above. The shift $M$ controls this sign
without changing the leading exponent.

\begin{lemma}[Endpoint flatness]\label{lem:power-flatness}
Under the notation and assumptions of \cref{lem:power-leading}, choose
$M\ge\alpha+\rho/(sL)$. Then the same loss satisfies $H(p)$, and hence
$H(\beta)$ for every $2<\beta\le p$.
\end{lemma}

\begin{proof}[Proof of \cref{lem:power-flatness}]
\proofstep{1}{Retain the first correction in the nodal data}
Write $K_\alpha=\alpha(\alpha-1)/2$. With $M$ fixed, the gamma-ratio
expansion \citep[Eq.~5.11.13]{NIST:DLMF} and Taylor expansion give
\[
 \frac{\Gamma(k+\alpha)}{\Gamma(k)}
 =k^\alpha\left(1+\frac{K_\alpha}{k}+O(k^{-2})\right),
 \qquad
 (k+M)^{-\rho}=k^{-\rho}
       \left(1-\frac{\rho M}{k}+O(k^{-2})\right).
\]
Multiplying these expressions, and inverting the gamma ratio for the
prefactor in the velocity formula, yields
\begin{equation*}
\begin{aligned}
 \frac{\Gamma(k+\alpha)}{\Gamma(k)}(k+M)^{-\rho}
 &=k^a\left(1+\frac{K_\alpha-\rho M}{k}+O(k^{-2})\right),\\
 \frac{\Gamma(n)}{\Gamma(n+\alpha)}
 &=n^{-\alpha}\left(1-\frac{K_\alpha}{n}+O(n^{-2})\right).
\end{aligned}
\end{equation*}
Because $a=\alpha-\rho>0$, Euler--Maclaurin summation
\citep[Eq.~2.10.1]{NIST:DLMF} gives
\begin{align*}
 \sum_{k=1}^n k^a
 &=\frac{n^{a+1}}{a+1}+\frac12n^a+O(1+n^{a-1}),\\
 \sum_{k=1}^n k^{a-1}&=\frac{n^a}{a}+o(n^a),\\
 \sum_{k=1}^n O(k^{a-2})
 &=\begin{cases}
 O(1),&0<a<1,\\
 O(\log n),&a=1,\\
 O(n^{a-1}),&a>1.
 \end{cases}
\end{align*}
For the first line, the integral and upper-endpoint terms are displayed;
the lower-endpoint and derivative-remainder terms are
$O(1+n^{a-1})$. The second line follows by integral comparison,
and the third by summing the absolute remainder bound.
Every remainder is $o(n^a)$ for the indicated range of $a$.
Adding the leading and first-correction terms therefore yields
\[
 \sum_{k=1}^n\frac{\Gamma(k+\alpha)}{\Gamma(k)}(k+M)^{-\rho}
 =\frac{n^{a+1}}{a+1}
  +\left(\frac12+\frac{K_\alpha-\rho M}{a}\right)n^a+o(n^a).
\]
Multiplying this sum by
$sn^{-\alpha}(1-K_\alpha/n+O(n^{-2}))$ gives the leading coefficient
$B=s/(a+1)$ and the relative $n^{-1}$ coefficient
$(a+1)[1/2+(K_\alpha-\rho M)/a]-K_\alpha$. Thus
\begin{equation}
 v_{n+1}=Bn^{-\delta-1}
   \left(1+\frac{V_1}{n}+o(n^{-1})\right),\qquad
 V_1=(a+1)\left(\frac12+\frac{K_\alpha-\rho M}{a}\right)-K_\alpha.
 \label{eq:power-first-velocity}
\end{equation}
For the position correction the starting index of the tail matters.
Because $x_{n+1}=\sum_{m=n+1}^\infty v_{m+1}$, use
\begin{align*}
 \sum_{m=n+1}^{\infty}m^{-\delta-1}
 &=\frac{n^{-\delta}}\delta-\frac12n^{-\delta-1}
       +O(n^{-\delta-2}),\\
 \sum_{m=n+1}^{\infty}m^{-\delta-2}
 &=\frac{n^{-\delta-1}}{\delta+1}+O(n^{-\delta-2}).
\end{align*}
These are the tail forms of Euler--Maclaurin; the minus sign in the
first endpoint correction reflects exclusion of $m=n$.
The velocity remainder is $o(m^{-\delta-2})$, whose tail is
$o(n^{-\delta-1})$: for any $\varepsilon>0$, its absolute value is
eventually bounded by $\varepsilon m^{-\delta-2}$, and the second
tail estimate applies. Hence
\[
 x_{n+1}=\frac B\delta n^{-\delta}
  +B\left(-\frac12+\frac{V_1}{\delta+1}\right)n^{-\delta-1}
  +o(n^{-\delta-1}).
\]
Adding $sg_n=O(n^{-\delta-2})=o(n^{-\delta-1})$ to get $y_n$
and factoring out $An^{-\delta}$, we obtain
\begin{equation}
 y_n=An^{-\delta}\left(1+\frac{Y_1}{n}+o(n^{-1})\right),\qquad
 Y_1=-\frac\delta2+\frac{\delta V_1}{\delta+1}.
 \label{eq:power-first-position}
\end{equation}
Since $p-1=\rho/\delta$, Taylor expansion of the denominator gives
\[
 y_n^{p-1}=A^{p-1}n^{-\rho}
    \left(1+\frac{(p-1)Y_1}{n}+o(n^{-1})\right).
\]
Divide $g_n=n^{-\rho}(1-\rho M/n+O(n^{-2}))$ by this expression,
using $(1+t)^{-1}=1-t+O(t^2)$ as $t\to0$. The resulting coefficient
is $H_1=-\rho M-(p-1)Y_1$, so
\begin{equation}
 \frac{g_n}{y_n^{p-1}}
 =A^{-(p-1)}\left(1+\frac{H_1}{n}+o(n^{-1})\right),
 \qquad H_1=-\rho M+\frac\rho2-\frac{\rho V_1}{\rho-1}.
 \label{eq:power-nodal-correction}
\end{equation}
\proofstep{2}{Choose the sign through the shift}
First simplify the velocity coefficient to
$V_1=(a+1)/2+K_\alpha/a-\rho M(a+1)/a$.
Then
\begin{align*}
 \frac{H_1}{\rho}
 &=\left[-1+\frac{\rho(a+1)}{a(\rho-1)}\right]M
   +\frac12-\frac{a+1}{2(\rho-1)}
   -\frac{K_\alpha}{a(\rho-1)}\\
 &=\frac{\alpha M}{a(\rho-1)}
   +\frac{2\rho-\alpha-2}{2(\rho-1)}
   -\frac{\alpha(\alpha-1)}{2a(\rho-1)}.
\end{align*}
Here the coefficient of $M$ uses $a+\rho=\alpha$.
Factoring out $\alpha/[a(\rho-1)]$ proves
\[
 H_1=\frac{\alpha\rho}{a(\rho-1)}
 \left[M-\frac{\alpha-1}{2}
           -\frac{(\alpha-\rho)(\alpha+2-2\rho)}{2\alpha}\right].
\]
The coefficient of $M$ is positive. To bound the two subtracted terms,
if $\alpha+2-2\rho\le0$ their sum is at most $(\alpha-1)/2<\alpha$.
Otherwise, $0<a<\alpha$ and $\rho>2$ imply
\[
 0<a(\alpha+2-2\rho)<\alpha(\alpha-2),
 \qquad
 \frac{\alpha-1}{2}
 +\frac{a(\alpha+2-2\rho)}{2\alpha}
 <\frac{2\alpha-3}{2}<\alpha.
\]
Thus in either case
$M\ge\alpha+\rho/(sL)$ ensures both $H_1>0$ and the smoothness bound
from \cref{lem:power-leading}.

\proofstep{3}{Show that interpolation preserves the correction}
A nodal inequality alone would not establish $H(p)$ for every nearby
point. We therefore compare the affine interpolant with the smooth
function described by the first two terms of the nodal expansion. Set
\[
 \varepsilon=\delta^{-1},\qquad
 C_0=A^{-(p-1)},\qquad c=H_1A^{-\varepsilon}>0.
\]
Raising $y_n=An^{-\delta}(1+O(n^{-1}))$ to the power
$\varepsilon=1/\delta$ gives
$y_n^\varepsilon=A^\varepsilon n^{-1}(1+O(n^{-1}))$.
Thus $n^{-1}=A^{-\varepsilon}y_n^\varepsilon(1+o(1))$, and
\eqref{eq:power-nodal-correction} becomes
\[
 g_n=C_0y_n^{p-1}
       \left(1+c y_n^\varepsilon+o(y_n^\varepsilon)\right).
\]
Let $h(x)=C_0x^{p-1}(1+cx^\varepsilon)$ for $x>0$.
On the interval $I_n=[y_{n+1},y_n]$, the leading mesh estimate gives
\[
 \frac{|I_n|}{y_n}
 =\frac{y_n-y_{n+1}}{y_n}\sim\frac{B}{An}=\frac\delta n,
 \qquad \frac{y_{n+1}}{y_n}\to1.
\]
Write $e_n=g_n-h(y_n)=o(y_n^{p-1}/n)$. At
$x=(1-t)y_{n+1}+ty_n$, $0\le t\le1$, let
$\mathcal I_nh(x)=(1-t)h(y_{n+1})+th(y_n)$.
Since $\psi$ uses the same affine weights,
\[
 |\psi(x)-\mathcal I_nh(x)|
 \le\max\{|e_n|,|e_{n+1}|\}
 =o(y_n^{p-1}/n)
\]
uniformly on $I_n$. The last equality uses the adjacent-node ratio
and $(n+1)/n\to1$. Also
\[
 h''(x)=C_0(p-1)(p-2)x^{p-3}
 +C_0c(p-1+\varepsilon)(p-2+\varepsilon)x^{p-3+\varepsilon}.
\]
On each sufficiently late $I_n$, $y_n/2\le x\le y_n$, so
$\sup_{x\in I_n}|h''(x)|=O(y_n^{p-3})$, with an $n$-independent
constant. The classical linear interpolation remainder
\citep[Eq.~3.3.5]{NIST:DLMF} gives
$h(x)-\mathcal I_nh(x)=h''(\zeta_x)(x-y_{n+1})(x-y_n)/2$
for an intermediate point $\zeta_x\in I_n$.
Since $|(x-y_{n+1})(x-y_n)|\le |I_n|^2/4$, it bounds the error by
\[
 \frac18\sup_{x\in I_n}|h''(x)|\,|I_n|^2
 =O(y_n^{p-1}n^{-2})=o(y_n^{p-1}/n).
\]
Thus the interpolation error is one order smaller than the positive
correction. The estimate remains valid for $2<p<3$, when $h''$ may be
unbounded at zero, because it is used separately on each $I_n$,
where $x$ is comparable to $y_n$.
Combining the two error bounds gives
$\sup_{x\in I_n}|\psi(x)-h(x)|=o(y_n^{p-1}/n)$.
Uniformly on $I_n$, $x/y_n\to1$ and
$nx^\varepsilon\to A^\varepsilon>0$; dividing by
$x^{p-1+\varepsilon}$ therefore makes this error tend to zero.
It follows, uniformly as $x\downarrow0$, that
\begin{equation}
 \psi(x)=C_0x^{p-1}
           \left(1+cx^\varepsilon+o(x^\varepsilon)\right).
 \label{eq:power-continuous-correction}
\end{equation}
\proofstep{4}{Integrate and verify the endpoint inequality}
To justify integration of the remainder, write
$\psi(x)=C_0x^{p-1}+C_0c x^{p-1+\varepsilon}
 +x^{p-1+\varepsilon}\eta(x)$, where $\eta(x)\to0$.
Then
\[
 \left|\int_0^x u^{p-1+\varepsilon}\eta(u)\,du\right|
 \le\sup_{0<u\le x}|\eta(u)|\,
       \frac{x^{p+\varepsilon}}{p+\varepsilon}
 =o(x^{p+\varepsilon}).
\]
Integrating the two explicit powers therefore gives
\[
 F(x)=\frac{C_0}{p}x^p
 \left(1+\frac{cp}{p+\varepsilon}x^\varepsilon
           +o(x^\varepsilon)\right).
\]
Divide the expansion of $xF'(x)$ by this expression, again using
$(1+t)^{-1}=1-t+O(t^2)$. The relative correction coefficient is
$c-cp/(p+\varepsilon)=c\varepsilon/(p+\varepsilon)>0$, so
\[
 \frac{xF'(x)}{F(x)}
 =p\left(1+\frac{c\varepsilon}{p+\varepsilon}x^\varepsilon
             +o(x^\varepsilon)\right)>p
\]
for all sufficiently small $x>0$: the $o(x^\varepsilon)$ term can
be bounded in absolute value by half of the displayed positive
correction. Evenness gives $xF'(x)=|x|F'(|x|)$ for $x<0$, proving
the same inequality
for negative $x$, while both sides of $pF(x)\le xF'(x)$ vanish at zero.
This proves $H(p)$. Since $F(x)\ge0$ and $\beta\le p$,
$\beta F(x)\le pF(x)\le xF'(x)$ on the same neighborhood, proving
every required $H(\beta)$.

\end{proof}

\subsection{Preserving the geometry in a prescribed dimension}

The quadratic transverse terms used for \cref{thm:no-gain} would violate
$H(p)$ when $p>2$. Here the extra coordinates must instead have local
$p$-power geometry. Capping their derivatives away from zero enforces
global smoothness while retaining coercivity.

\begin{lemma}[Finite-dimensional power witness]\label{lem:power-embedding}
Let $F_1$ be the scalar loss from \cref{lem:power-flatness}. For every
fixed $d\ge1$, there is a coercive convex loss $\widetilde F:\R^d\to\R$
with $L$-Lipschitz gradient, unique minimizer zero, and local conditions
$L(p)$ and $H(p)$, whose exact CD orbit from a suitable initialization
satisfies $\widetilde F(z_n)\sim Dn^{-2p/(p-2)}$ with the same $D>0$.
\end{lemma}

\begin{proof}[Proof of \cref{lem:power-embedding}]
The case $d=1$ is \cref{lem:power-leading,lem:power-flatness}.
Choose $r_1,K_1>0$ so that both $F_1(t)\ge K_1|t|^p$ and
$tF_1'(t)\ge pF_1(t)$ hold for $|t|\le r_1$.
For $d>1$, choose $R,\kappa>0$ with
$\kappa(p-1)R^{p-2}\le L$, and define $\varphi(0)=0$ and
\[
 \varphi'(t)=\kappa\,\operatorname{sgn}(t)
               \min\{|t|^{p-1},R^{p-1}\},\qquad
 \widetilde F(z)=F_1(z_1)+\sum_{i=2}^d\varphi(z_i).
\]
For example, $R=1$ and $\kappa=L/(p-1)$ meet this condition.
Integrating the derivative gives the explicit extension
\[
 \varphi(t)=
 \begin{cases}
 \kappa |t|^p/p,&|t|\le R,\\
 \kappa R^p/p+\kappa R^{p-1}(|t|-R),&|t|\ge R.
 \end{cases}
\]
Its derivative is continuous and nondecreasing. Away from the two joins,
its slope is $\kappa(p-1)|t|^{p-2}\le L$ inside $[-R,R]$ and zero
outside. Splitting an interval at the joins proves the global
$L$-Lipschitz bound. Thus $\varphi$ is convex, coercive, nonnegative,
and vanishes only at zero.
For any prescribed height $H>0$, choose a common radius beyond which
each of the finitely many summands exceeds $H$. If $\|z\|_2$ exceeds
$\sqrt d$ times that radius, at least one coordinate does, so
$\widetilde F(z)>H$. This proves coercivity of the sum; its unique
minimizer is zero because every summand is nonnegative and vanishes only
there. Since $F_1'$ and $\varphi'$ are both
$L$-Lipschitz, for all $z,w\in\R^d$,
\[
 \|\nabla\widetilde F(z)-\nabla\widetilde F(w)\|^2
 \le L^2\sum_{i=1}^d|z_i-w_i|^2=L^2\|z-w\|^2.
\]
For $\|z\|_2\le\min\{r_1,R\}$, all scalar local inequalities apply,
and $z_i\varphi'(z_i)=p\varphi(z_i)$ for $i\ge2$.
The norm comparison follows from convexity of $t\mapsto t^{p/2}$:
\[
 \left(\frac1d\sum_{i=1}^d|z_i|^2\right)^{p/2}
 \le\frac1d\sum_{i=1}^d|z_i|^p,
 \qquad
 \sum_{i=1}^d|z_i|^p\ge d^{1-p/2}\|z\|_2^p.
\]
Combining the scalar growth bounds with this comparison, and summing
the scalar flatness inequalities, gives
\[
 \widetilde F(z)
 \ge d^{1-p/2}\min\{K_1,\kappa/p\}\|z\|_2^p,
 \qquad
 \langle\nabla\widetilde F(z),z\rangle\ge p\widetilde F(z).
\]
Consequently, $L(p)$ and $H(p)$ hold in every prescribed finite dimension.
Initialize $z_0=z_1=(x_1,0,\ldots,0)$, where $x_1$ is the scalar
initial point. If two consecutive values of coordinate $i\ge2$ are
zero, its extrapolated value is zero and $\varphi'(0)=0$ keeps the next
value zero. Induction therefore gives $z_n=(x_n,0,\ldots,0)$ and
$\widetilde F(z_n)=F_1(x_n)$ for all $n$. Equation
\eqref{eq:power-nodal-asymptotic} and the interlacing argument therefore
give the claimed exact asymptotic.
\end{proof}

\subsection{Completion of the sharpness statement}

For either parameter regime in Appendix~\ref{app:targets}, the choice
$\rho=2+2/(p-2)$ satisfies $2<\rho<\alpha$, as verified in
\eqref{eq:power-parameter-range}. Choose
$M\ge\alpha+\rho/(sL)$ and apply
\cref{lem:power-leading,lem:power-flatness,lem:power-embedding}.
The resulting loss has all the required properties. In regime~(i),
$H(p)$ implies the prescribed $H(\beta)$; in regime~(ii), the unique
minimizer makes $L(p)$ identical to $L_{\mathrm{set}}(p)$.
For every gain $G(n)\to\infty$, the exact asymptotic gives
\[
 n^{2p/(p-2)}G(n)\bigl(\widetilde F(z_n)-\widetilde F^*\bigr)
 \sim DG(n)\longrightarrow\infty.
\]
Thus no divergent gain improves the rate, proving the sharpness
assertion of \cref{thm:power}. The upper bounds
for the whole class are established in Appendix~\ref{app:upper-transfer}; their
proofs are independent of this construction.

\clearpage
\subsection{Illustrating the two constructions}
\label{app:construction-instances}

\Cref{fig:construction-instances} shows concrete instances of the block
and power constructions developed in Appendices~\ref{app:inverse} and~\ref{app:power}.
Each row follows the same sequence: prescribe the queried gradients,
recover the query points and interpolate the derivative, then integrate
to obtain the loss. The plots illustrate the functions whose properties
were proved above. The figure generator records the recurrence checks
and bounds on the omitted integration tails; Appendix~\ref{app:supplement}
describes the accompanying code and outputs.

\begin{figure}[H]
\centering
\includegraphics[width=\linewidth]{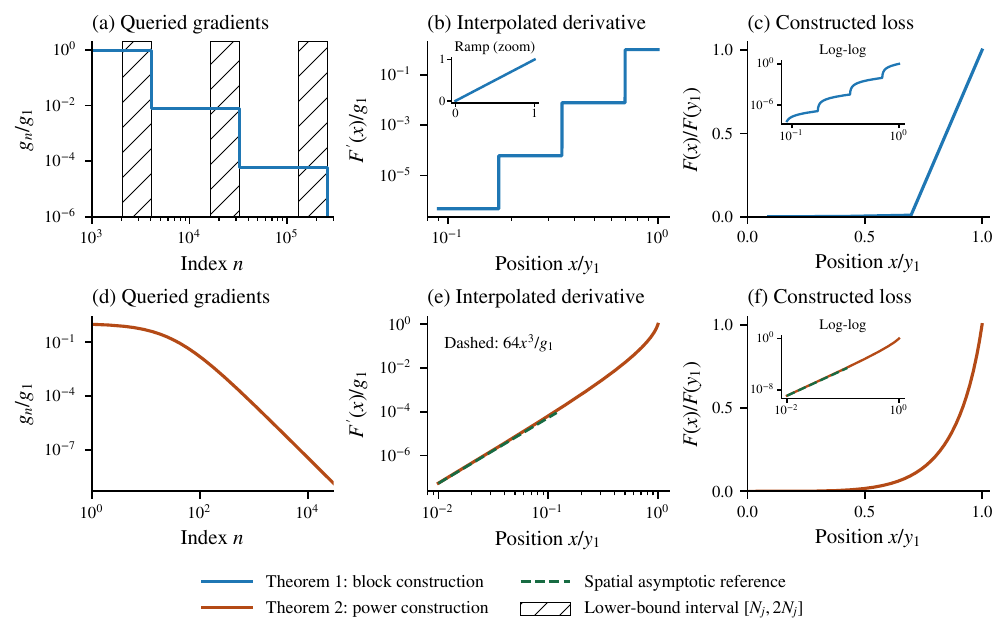}
\caption{Two choices of gradients and the losses they produce, with
$\alpha=4$, $s=1/2$, and $L=1$. Top: a no-gain instance for $G(n)=n$,
using $N_j=2048\cdot8^{j-1}$ and $a_j=2^{-j}$. Hatching marks the
intervals $[N_j,2N_j]$ used in the lower bound. The inset in (b) magnifies one
continuous affine join in $F'$. Bottom: the power instance with $p=4$
and $M=32$, in the flatness regime of \cref{thm:power}.
Panels (c) and (f) use linear axes to display convexity; their insets
retain log--log views of the same curves. The other main panels use
log--log axes. Dashed curves in (e) and the inset of (f) show the spatial
asymptotics $F'(x)\sim64x^3$ and $F(x)\sim16x^4$, respectively.}
\label{fig:construction-instances}
\end{figure}

\FloatBarrier

\section{Local-geometry upper bounds}\label{app:upper-transfer}

This section proves that the power instances attain rates valid for
every objective in the stated classes. The rate mechanisms come from
\citet{Apidopoulos2021} and \citet{aujol2024strongconvergencefistaiterates}. Our technical extension is to remove
the coercivity assumption in the growth-only result in finite dimension.
We first match the iteration conventions and establish eventual local
geometry, then give the two upper-bound arguments separately.

\subsection{Matching the iteration conventions}

The upper-rate analyses of \citet{Apidopoulos2021} and \citet{aujol2024strongconvergencefistaiterates} use a
different iteration index. We shift our index to match their recurrence
and apply the published energy estimates to the same CD orbit.

Let the setting and recurrence be those of Appendix~\ref{app:setting}. Let $T_s=I-s\nabla F$, set
$z_{-1}=x_1$, and put $z_m=x_{m+1}$ for $m\ge0$. Since
$\beta_{m+1}=m/(m+\alpha)$, for $m\ge1$ we have
$y_{m+1}=z_m+[m/(m+\alpha)](z_m-z_{m-1})$ and
$z_{m+1}=T_s(y_{m+1})$. For $m=0$, both expressions reduce to
$y_1=x_1=z_0$ and $z_1=T_s(z_0)$ because $\beta_1=0$.
Thus at every index,
\begin{equation}
 z_{m+1}=T_s\!\left(z_m+\frac{m}{m+\alpha}(z_m-z_{m-1})\right),
 \qquad m\ge0.
 \label{eq:published-index-bridge}
\end{equation}
The coefficient is exactly $m/(m+\alpha)$, as in the cited upper-rate
arguments; in \citet{Apidopoulos2021}, the damping parameter is denoted $b$,
so $b=\alpha$. In particular, the gradient is evaluated at the
extrapolated point inside $T_s$. Our recurrence starts at $m=0$ with
$z_{-1}=z_0$ and zero inertial coefficient. The subsequent energy
inequalities apply to the resulting finite initial states $z_0,z_1$;
their validity does not require these states to coincide. Because
$\beta_1=0$, an arbitrary $x_0$ has no effect on the original iterates
beginning with $x_1$.

\subsection{Local geometry along the convergent orbit}

The local assumptions need only hold eventually, but this must be
justified for both the iterates and the extrapolated query points.
Finite-dimensional convergence provides that localization even when
the objective is not coercive.

\begin{lemma}[Eventual local geometry]\label{lem:localization}
Let $F:\R^d\to\R$ be convex, attain its minimum, and have a globally
$L$-Lipschitz gradient, where $d\ge1$ is finite and $L>0$.
For $\alpha>3$ and a fixed step $0<s<L^{-1}$, let $(x_n)$ and $(y_n)$
be generated by the CD iteration \eqref{eq:app-cd} from any
$x_0,x_1\in\R^d$. Then there is a minimizer $\bar x$ such that
$x_n\to\bar x$ and $y_n\to\bar x$ in Euclidean norm.
Consequently, if $H(\beta)$ and $L(p)$ hold at the unique minimizer,
they hold at all sufficiently late iterates and query points.
If $L_{\mathrm{set}}(p)$ holds, its inequality likewise applies at
all sufficiently late iterates and query points, without uniqueness
or coercivity.
\end{lemma}

\begin{proof}[Proof of \cref{lem:localization}]
The CD convergence theorem \citep{Chambolle2015} gives weak convergence of
$(x_n)$ to a minimizer for $\alpha>3$. In finite dimension this is
norm convergence. Since $x_n-x_{n-1}\to0$ and $0\le\beta_n<1$,
the query points satisfy
\begin{align*}
 \|y_n-\bar x\|
 &\le\|x_n-\bar x\|+\beta_n\|x_n-x_{n-1}\|\\
 &\le2\|x_n-\bar x\|+\|x_{n-1}-\bar x\|\longrightarrow0.
\end{align*}
For a unique minimizer, both sequences therefore enter and remain
in each neighborhood in \eqref{eq:app-H}--\eqref{eq:app-Lp}.
In the set-valued case,
$\dist(x_n,X^*)\le\|x_n-\bar x\|\to0$, and the same holds for
$y_n$. This proves eventual validity of \eqref{eq:app-Lset}.
\end{proof}

We apply the published asymptotic energy inequalities from this eventual
index onward. The original momentum coefficients $m/(m+\alpha)$ are
retained throughout; localization does not restart the algorithm.
A finite prefix affects only the constants in the estimates.

\subsection{Upper rate under flatness and growth}

Numbered references to \citet{Apidopoulos2021} in this subsection use the authors'
\href{https://hal.science/hal-01965095v3/document}{HAL version~3}.
For a unique minimizer, their Theorem~3.2 uses the local conditions
\eqref{eq:app-H}--\eqref{eq:app-Lp}. In our notation it gives
$F(z_m)-F^*=O(m^{-2p/(p-2)})$ whenever $p\ge\beta>2$ and
\[
 \alpha\ge\frac{\beta+2}{\beta-2}.
\]
That result uses the same exponents $\beta,p$, and its step $\gamma$
is our $s$.
The result allows $0<s\le1/L$, so our strict step restriction is
included. Its damping threshold includes equality, as also recorded in
\citet[Table~1]{aujol2024strongconvergencefistaiterates}. By \cref{lem:localization}, the local
inequalities are valid along the required tail. The cited proof uses
an eventual one-step energy estimate (their Lemma~4.3 and equation~(4.28)),
then sums it from an index at which the local conditions hold. That
one-step estimate depends on the current states and the coefficient
$m/(m+\alpha)$, not on equality of the two initial states. Its initial
energy is finite for our $z_0,z_1$, so the same argument applies to the
initialization in \eqref{eq:published-index-bridge}.

To spell out the final index conversion, put $q=2p/(p-2)$ and choose
$C,N$ such that $F(z_m)-F^*\le Cm^{-q}$ for $m\ge N$.
For $n\ge\max\{N+1,2\}$,
\[
 F(x_n)-F^*=F(z_{n-1})-F^*
 \le C(n-1)^{-q}\le 2^q Cn^{-q},
\]
because $n-1\ge n/2$. This proves part~(i)
of \cref{thm:power}.

\subsection{Upper rate under growth alone, without coercivity}

Here the minimizer set may have more than one point. The energy uses
distance to that set and does not require flatness. There are two issues
to check: how local growth becomes available without coercivity, and how
to extract a residual bound from an energy with a negative distance term.

\proofstep{1}{Replace coercivity by convergence in the local-growth step}
Set $\widehat L=1/s$. Since the gradient is $L$-Lipschitz and
$L<\widehat L$, it is also $\widehat L$-Lipschitz. The step
$s=1/\widehat L$ is the one used in
\citet[Theorem~1]{aujol2024strongconvergencefistaiterates}, with nonsmooth term zero. Their coercivity
assumption is used in Lemma~4, through Lemma~1, to obtain eventual validity
of the local growth inequality along the orbit. \Cref{lem:localization} already gives that conclusion for the
present finite-dimensional CD orbit. The
following calculation records the remainder of the rate argument without
coercivity.

\proofstep{2}{Relate distance to residual and use the published energy estimate}
Let $z_m^*$ be the Euclidean projection of $z_m$ onto the closed convex set
$X^*$, and define
\[
 w_m=\frac{2}{\widehat L}\bigl(F(z_m)-F^*\bigr),\qquad
 h_m=\|z_m-z_m^*\|^2,\qquad
 b_m=\frac{m}{m+\alpha},\qquad
 \theta=1+\frac4{p-2}.
\]
By $L_{\mathrm{set}}(p)$ and $\dist(z_m,X^*)\to0$, there is an index $N$
such that $K h_m^{p/2}\le F(z_m)-F^*=(\widehat L/2)w_m$ for
$m\ge N$. Both sides are nonnegative, so taking the power $2/p$ gives
\begin{equation}
 h_m\le C_K w_m^{2/p}\quad(m\ge N),\qquad
 C_K=\left(\frac{\widehat L}{2K}\right)^{2/p}.
 \label{eq:eventual-set-growth}
\end{equation}
This is the estimate required in equations~(49)--(50) of
\citet{aujol2024strongconvergencefistaiterates}.

Put $\lambda=\alpha-1-\theta>0$ and
$\xi=\lambda(\lambda+1-\alpha)=-\lambda\theta<0$, and consider
\begin{align*}
 E_m={}&m^2w_m+
 \|\lambda(z_m-z_m^*)+mb_m(z_m-z_{m-1})\|^2+\xi h_m\\
 &\quad+\lambda m b_m^2\|z_m-z_{m-1}\|^2,
 \qquad J_m=m^\theta E_m.
\end{align*}
We use the one-step energy estimate of
\citet[Lemma~6, Eq.~(52)]{aujol2024strongconvergencefistaiterates} as a cited input. To verify its
applicability and the signs needed below, first note the two elementary
inequalities underlying that estimate. If $u^+=y-s\nabla F(y)$,
smoothness with constant $\widehat L=1/s$ and convexity imply, for any $x$,
\begin{align*}
 F(u^+)-F(x)
 &\le\langle\nabla F(y),y-x\rangle
        -\frac s2\|\nabla F(y)\|^2\\
 &=\frac1{2s}\bigl(\|y-x\|^2-\|u^+-x\|^2\bigr).
\end{align*}
Here the equality follows by expanding
$\|y-s\nabla F(y)-x\|^2$.
Using $y=z_m+b_m(z_m-z_{m-1})$, first with $x=z_m$ and then with
$x=z_m^*$, gives respectively
\begin{align*}
 w_{m+1}-w_m
 &\le b_m^2\|z_m-z_{m-1}\|^2-\|z_{m+1}-z_m\|^2,\\
 w_{m+1}
 &\le\|y-z_m^*\|^2-\|z_{m+1}-z_m^*\|^2.
\end{align*}
Also, projection onto the closed convex minimizer set gives
$\langle z_m-z_m^*,v-z_m^*\rangle\le0$ for every $v\in X^*$.
This follows by minimizing squared distance along
$z_m^*+t(v-z_m^*)$ and taking the right derivative at $t=0$.
In particular, with
\[
 \Delta_m^*=z_{m+1}^*-z_m^*,\qquad
 Q_m=\|\Delta_m^*\|^2
       -2\langle z_m-z_m^*,\Delta_m^*\rangle,
\]
we have $Q_m\ge0$. These facts require neither coercivity nor a
unique minimizer.

To state the cited estimate without confusing its notation with ours,
put $U_m=\|\lambda(z_m-z_m^*)+mb_m(z_m-z_{m-1})\|^2$.
The source's growth exponent $\gamma$ is our $p$, its energy weight
$p$ is our $\theta$, its inertial coefficient $\alpha_m$ is our $b_m$,
and its squared-norm term $b_m$ is our $U_m$.
Its estimate reads
\begin{align*}
 J_{m+1}-J_m\le{}&
  [C_1(m+1)^{\theta+1}+R_1(m)]w_{m+1}\\
 &+[C_2(m+1)^{\theta-1}+R_2(m)]U_{m+1}\\
 &+[C_3(m+1)^{\theta-1}+R_3(m)]h_{m+1}
   -m^\theta B_4(m)Q_m,
\end{align*}
where $R_1(m)=O(m^\theta)$,
$R_2(m),R_3(m)=O(m^{\theta-2})$, and substitution of
$\lambda=\alpha-1-\theta$ gives
\begin{align*}
 C_1&=2-\lambda+\theta=3+2\theta-\alpha<0,\\
 C_2&=2(\lambda+1-\alpha)+\theta=-\theta<0,\\
 C_3&=\lambda(\lambda+1-\alpha)(\theta-2\lambda)
      =\lambda\theta(2\lambda-\theta)>0,\\
 B_4(m)&=2\lambda\theta-
              \frac{\alpha^2\lambda}{m+1+\alpha}
              \longrightarrow2\lambda\theta>0.
\end{align*}
The damping assumption is exactly
$\alpha>3+2\theta=5+8/(p-2)$; it gives
$\lambda>\theta+2>0$, so all the displayed signs follow.
Since each remainder is of lower order than its leading coefficient,
enlarge $N$ until
\begin{align*}
 C_1(m+1)^{\theta+1}+R_1(m)&\le\tfrac12C_1(m+1)^{\theta+1},\\
 C_2(m+1)^{\theta-1}+R_2(m)&\le0,\\
 C_3(m+1)^{\theta-1}+R_3(m)&\le2C_3(m+1)^{\theta-1},
 \qquad B_4(m)\ge0.
\end{align*}
The $U_{m+1}$ and $Q_m$ terms are then nonpositive and can be dropped.
With $c_0=-C_1/2>0$ and $c_1=2C_3>0$, this proves
\begin{equation}
 J_{m+1}-J_m
 \le-c_0(m+1)^{\theta+1}w_{m+1}
       +c_1(m+1)^{\theta-1}h_{m+1}.
 \label{eq:growth-transfer-energy}
\end{equation}
This is the source's equation~(57), with its hypotheses and coefficient
choices made explicit.

\proofstep{3}{Bound the energy increments}
The exponent identity needed for the substitution is
\[
 \theta-1=\frac4{p-2},\qquad
 \theta+1=2+\frac4{p-2}=\frac{2p}{p-2},\qquad
 \frac{2(\theta+1)}p=\theta-1.
\]
Thus substituting
\eqref{eq:eventual-set-growth} into
\eqref{eq:growth-transfer-energy} bounds its right side by
$-c_0u+c_1C_Ku^{2/p}$, where
$u=(m+1)^{\theta+1}w_{m+1}$. In detail,
$(m+1)^{\theta-1}w_{m+1}^{2/p}=u^{2/p}$ by the displayed identity.
To bound the scalar expression explicitly, let
$\nu=2/p\in(0,1)$, $A_0=c_1C_K>0$, and
$T_0=(2A_0/c_0)^{1/(1-\nu)}$.
For $u\ge T_0$, $A_0u^\nu\le(c_0/2)u$ and the expression is
nonpositive; for $0\le u\le T_0$, it is at most
$M_0:=A_0T_0^\nu$. Therefore
\[
 J_m=J_N+\sum_{j=N}^{m-1}(J_{j+1}-J_j)
 \le J_N+M_0(m-N)\le C m\qquad(m\ge N),
\]
where $N\ge1$ and, for example, $C=M_0+|J_N|+1>0$ suffices.
\proofstep{4}{Control the negative term and recover the residual rate}
An upper bound on $J_m$ alone is not yet a residual bound because
$\xi<0$. All terms in $E_m$ other than $\xi h_m$ are nonnegative, so
dividing $J_m\le Cm$ by $m>0$ and using $J_m=m^\theta E_m$ gives
\[
 m^{\theta+1}w_m-|\xi|m^{\theta-1}h_m\le C.
\]
Using \eqref{eq:eventual-set-growth} once more, we obtain
\[
 X_m-|\xi|C_K X_m^{2/p}\le C,
 \qquad X_m=m^{\theta+1}w_m.
\]
Indeed, $h_m\le C_Kw_m^\nu$ implies
$-\lvert\xi\rvert m^{\theta-1}h_m
\ge-\lvert\xi\rvert C_K X_m^\nu$, so replacing the negative term
preserves the stated upper inequality for the smaller expression.
To make boundedness explicit, put $B_0=|\xi|C_K$ and
$V_0=(2B_0)^{1/(1-\nu)}$. If $X_m\ge V_0$, then
\[
 B_0X_m^\nu\le\tfrac12X_m,
 \qquad \tfrac12X_m\le X_m-B_0X_m^\nu\le C.
\]
Consequently $X_m\le\max\{V_0,2C\}$ in both cases. Since
$F(z_m)-F^*=(\widehat L/2)w_m$, we conclude
\[
 F(z_m)-F^*=O(m^{-\theta-1})
           =O(m^{-2p/(p-2)}).
\]
This reproduces the rate step in equations~(58)--(62) of
\citet{aujol2024strongconvergencefistaiterates} with eventual growth supplied by finite-dimensional CD
convergence. It proves part~(ii), including noncoercive objectives and
non-singleton minimizer sets. Finally, the fixed shift $z_m=x_{m+1}$
does not change either upper-rate exponent.

\section{The quadratic boundary}\label{app:quadratic}

The power constructions require $p>2$. To illustrate the different
behavior at the quadratic boundary, we give an explicit quadratic loss
and prove a geometric bound for its exact discrete CD orbit. The same
argument then covers every fixed convex quadratic in finite dimension.

\subsection{The loss and the claimed rate}

Fix $0<\lambda\le L$ and consider the scalar loss
\[
 F_\lambda(x)=\frac{\lambda}{2}x^2,\qquad x\in\R.
\]
It is smooth, convex, coercive, and uniquely minimized at zero, with
$F_\lambda^*=0$ and $F_\lambda'(x)=\lambda x$. Run CD with the same
damping $\alpha>3$ and step $0<s<L^{-1}$ as in the main results, from
$x_0=x_1=R>0$. We will show that, for every fixed
$\theta\in(1-s\lambda,1)$, there is a constant $C>0$ such that
\[
 F_\lambda(x_n)\le C\theta^n\qquad(n\ge1).
\]
This is the geometric decay claimed here: the error is bounded by an
exponentially decreasing envelope, even if it oscillates between steps.
For example, $F(x)=x^2/2$, $\alpha=4$, $s=1/2$, and $x_0=x_1=1$
give $F(x_n)=O((3/4)^n)$.

\subsection{Why the exact CD recurrence contracts}

For this loss, the gradient step simply multiplies the query point by
$a:=1-s\lambda\in(0,1)$. Substituting $F_\lambda'(y_n)=\lambda y_n$
into the CD update gives
\[
 x_{n+1}=a y_n
 =a\bigl((1+\beta_n)x_n-\beta_n x_{n-1}\bigr),
 \qquad \beta_n=\frac{n-1}{n+\alpha-1}.
\]
Because momentum involves two successive iterates, we track their joint
state $u_n=(x_n,x_{n-1})^\mathsf{T}$. It satisfies
\[
 u_{n+1}=M_nu_n,\qquad
 M_n=\begin{pmatrix}a(1+\beta_n)&-a\beta_n\\1&0\end{pmatrix}
 \longrightarrow
 M=\begin{pmatrix}2a&-a\\1&0\end{pmatrix}.
\]
The key is to find one fixed norm in which all sufficiently late updates
contract. This is a standard adapted-norm argument
\citep[Chapter~5]{Horn_Johnson_2012}; here the norm can be written explicitly:
\[
 \|(z,w)\|_a^2:=(1-a)z^2+a(z-w)^2.
\]
Since $0<a<1$, this positive definite quadratic form defines a norm.
Indeed, it is the Euclidean norm after the invertible linear map
$(z,w)\mapsto(\sqrt{1-a}\,z,\sqrt a\,(z-w))$.
For $u=(z,w)^\mathsf T$, substitute
$Mu=(2az-aw,z)^\mathsf T$ and expand:
\begin{align*}
 \|Mu\|_a^2
 &=(1-a)(2az-aw)^2+a((2a-1)z-aw)^2\\
 &=az^2-2a^2zw+a^2w^2
 =a\bigl((1-a)z^2+a(z-w)^2\bigr)=a\|u\|_a^2.
\end{align*}
Thus the limiting update contracts by exactly $\sqrt a<1$ in this norm.
To transfer this contraction to the varying matrices, observe that
\[
 (M_n-M)u=(a(\beta_n-1)(z-w),0)^\mathsf T.
\]
The norm of $(t,0)$ is $|t|$, and $a(z-w)^2\le\|u\|_a^2$.
Consequently,
\begin{align*}
 \|(M_n-M)u\|_a
 &\le\sqrt a\,(1-\beta_n)\|u\|_a
 =\frac{\alpha\sqrt a}{n+\alpha-1}\|u\|_a,\\
 \|M_nu\|_a
 &\le\left(\sqrt a+\frac{\alpha\sqrt a}{n+\alpha-1}\right)\|u\|_a.
\end{align*}
Fix $\theta\in(a,1)$ and put $r=\sqrt\theta>\sqrt a$.
Choose $N\ge1$ with
$\alpha\sqrt a/(N+\alpha-1)\le r-\sqrt a$.
Then $\|M_nu\|_a\le r\|u\|_a$ for every $u$ and $n\ge N$.
Applying this inequality successively gives
\[
 \|u_n\|_a\le r^{n-N}\|u_N\|_a\qquad(n\ge N).
\]
Finally, $(1-a)x_n^2\le\|u_n\|_a^2$, so the state contraction controls
the loss itself:
\[
 F_\lambda(x_n)=\frac{\lambda}{2}x_n^2
 \le\frac{\lambda\|u_N\|_a^2}{2(1-a)}\,
       \theta^{n-N}\qquad(n\ge N).
\]
For example, take $C$ to be the maximum of
$\lambda\|u_N\|_a^2\theta^{-N}/[2(1-a)]$ and the finite set
$\{F_\lambda(x_n)\theta^{-n}:1\le n\le N\}$.
Then the claimed bound holds for every $n\ge1$.

\subsection{Finite-dimensional quadratics and the scope of the argument}

More generally, take
\[
 F(x)=F^*+\frac12(x-x^*)^\mathsf{T}Q(x-x^*),
\]
where $Q$ is a fixed symmetric positive semidefinite matrix whose
eigenvalues $\lambda_i$ lie in $[0,L]$. In an orthonormal eigenbasis,
write $z_{i,n}$ for the coordinates of $x_n-x^*$. Orthogonal change of
coordinates turns $\nabla F(x)=Q(x-x^*)$ into coordinatewise
multiplication by $\lambda_i$. Thus
\[
 z_{i,n+1}=(1-s\lambda_i)
       \bigl((1+\beta_n)z_{i,n}-\beta_n z_{i,n-1}\bigr),
\]
and the objective is
\[
 F(x_n)-F^*=\frac12\sum_{i=1}^d\lambda_i z_{i,n}^2.
\]
Zero eigenvalues contribute no error. If $Q\ne0$, let $\lambda_+$ be
its smallest positive eigenvalue. For any
$\theta\in(1-s\lambda_+,1)$, the scalar argument bounds each positive
term by $C_i\theta^n$: indeed, for $\lambda_i>0$,
$a_i=1-s\lambda_i\le1-s\lambda_+<\theta$.
The scalar proof applies to any finite initial state $u_1$, so no sign
restriction on these coordinates is needed. Summing the finitely many
bounds gives
\[
 F(x_n)-F^*\le\left(\sum_{i:\lambda_i>0}C_i\right)\theta^n.
\]
If $Q=0$, the error is identically zero.

Consequently, fixed finite-dimensional quadratics cannot attain the
nonzero polynomial asymptotic of \cref{thm:power}. This calculation
concerns quadratic losses themselves; it does not classify all smooth
convex losses satisfying a local quadratic growth inequality.

\section{Forward CD experiments on the fixed witnesses}\label{app:numerical-witnesses}

The experiments run CD on the scalar losses constructed in
Appendices~\ref{app:inverse} and~\ref{app:power}. They display repeated exceedances of the
subsequence lower bound and agreement with the predicted power constant.
The infinitely-many-indices and arbitrary-gain statements are established
by the proofs; the numerical runs cover the finite horizons stated below.
Appendix~\ref{app:supplement} describes the accompanying code, sampled
traces, numerical checks, and reproduction commands.

\subsection{Fixed losses and independent gradient queries}

Every run uses $L=1$, $s=1/2$, and $x_0=x_1=1$. Starting from the
inverse construction, set $\lambda=1/x_1$ and replace the unscaled loss
by $\lambda^2F(\,\cdot\,/\lambda)$; below, $F$ denotes this rescaled
loss, whose minimum value is zero. The derivative interpolant is fixed
before optimization. Gradient evaluations use the current forward query
point; constant-gradient intervals also permit the batched updates
described below. The prescribed gradient sequence and inverse-orbit
positions are used only for building the loss and checking the resulting run.

The query mesh incorporates the infinite weighted gradient tail.
For the block instance a geometric sum captures its leading part, with
an explicit bound on the remaining correction. For the power instances
Euler--Maclaurin summation supplies the tail and a remainder bound.
Additional nodes below the queried region are used to
integrate the objective. If $y_{\rm cut}$ is the last integration node,
omitting the area below it gives the truncation estimate
\[
 0\le F(x)-\widehat F(x)\le
 y_{\rm cut}F'(y_{\rm cut}),\qquad x\ge y_{\rm cut}.
\]
In fact, exact integration over the retained affine pieces gives
$\widehat F(x)=\int_{y_{\rm cut}}^xF'(u)\,du$, so
\[
 F(x)-\widehat F(x)=\int_0^{y_{\rm cut}}F'(u)\,du.
\]
The bound follows from
$0\le F'(u)\le F'(y_{\rm cut})$ on $[0,y_{\rm cut}]$.
It controls truncation of the analytic integral; floating-point
roundoff is assessed separately by the precision comparisons below.
All gradient evaluations remain inside the retained mesh or its prescribed
constant outer extension. The implementation stops if a query leaves
this region toward zero. No completion of the loss below the retained
mesh is needed for the runs.

The power runs use extended precision with 64 mantissa bits, with a
second run using float64 optimizer states on the same oracle. Their
largest sampled relative position discrepancy from the inverse-orbit
reference is $1.23\times10^{-17}$, and the largest relative objective
difference between precisions is $1.02\times10^{-13}$. The block run
uses arbitrary precision and batches of the discrete recurrence, as
described below. Precision comparisons are measured consistency checks,
separate from the analytic tail bounds and not interval-arithmetic
certificates.

\subsection{Theorem 1: repeated exceedances with a double-logarithmic gain}

We take $G(n)=\log\log(e^e+n)$ and $\alpha=4$. There is no slowest
divergent gain; this choice grows more slowly than every positive power
of $\log n$. To display several episodes, use the summable amplitudes
$a_j=(1-r)r^{j-1}$ with $r=0.999$, in the same block construction:
\[
 N_0=0,\qquad
 N_j=2^{\left\lceil\exp\!\bigl(2j r^{-2(j-1)}\bigr)/\log2\right\rceil},
 \qquad g_k=\frac{a_j}{N_j^2}\quad(2N_{j-1}<k\le2N_j).
\]
The weighted-summability and transition-slope estimates in
Appendix~\ref{app:no-gain-proof} apply unchanged to these amplitudes.
The scales satisfy the separation and first-scale conditions, and
$G(N_j)a_j^2\ge2j(1-r)^2$ for every $j$. Thus the single fixed loss
has the infinite sequence of bounds
\[
 \frac{N_j^2G(N_j)F(x_{N_j})}{K}\ge2j,
 \qquad K=\lambda^2\frac{se^{-\alpha}}6(1-r)^2>0.
\]
Here is the verification of the scale conditions. Write
$b_j=2j r^{-2(j-1)}$ and $q_j=e^{b_j}/\log2$, so that
$N_j=2^{\lceil q_j\rceil}$. Then
\[
 b_{j+1}-b_j
 =2r^{-2j}\bigl((j+1)-jr^2\bigr)>2,
 \qquad q_{j+1}>e^2q_j.
\]
Since $q_j\ge q_1=e^2/\log2>10$, this implies
$\lceil q_{j+1}\rceil-\lceil q_j\rceil\ge2$ and hence
$N_{j+1}\ge4N_j>2N_j$.
The first scale is $N_1=2^{11}=2048>30e^4
=3(\alpha+1)e^\alpha/(sL)$ for the stated parameters.
Moreover, $\log N_j\ge e^{b_j}$ implies
\[
 G(N_j)=\log\log(e^e+N_j)\ge\log\log N_j\ge b_j,
 \qquad G(N_j)a_j^2\ge2j(1-r)^2.
\]
The amplitudes satisfy $\sum_j a_j=1$ and
$c_{j+1}/c_j=r(N_j/N_{j+1})^2<1$, so the summability and monotonicity
arguments still apply. Multiplying the block lower bound by the
rescaling factor $\lambda^2$ and dividing by $K$ yields the displayed
bound $2j$.
The plotted reference is $K/[n^2G(n)]$; $K$ is determined by the
construction, without fitting the observed errors.

The first five selected indices are
\[
 (N_1,N_2,N_3,N_4,N_5)
 =(2^{11},2^{80},2^{597},2^{4513},2^{34437}).
\]
Since $\log_{10}N_5\approx10366.57$, executing every step separately
is impractical. We instead sum the discrete recurrence exactly on each
constant-gradient interval. For $g_n=c$ on $m\le n<k$, with $2\le m<k$, let
$A=sc/5$ and $H=(v_m-A(m+3))(m-1)m(m+1)(m+2)$. Then
\[
 \begin{aligned}
 v_k&=A(k+3)+\frac{H}{(k-1)k(k+1)(k+2)},\\
 x_k&=x_m-\frac A2(k-m)(k+m+7)
 -\frac H3\left(\frac1{m(m+1)(m+2)}
              -\frac1{k(k+1)(k+2)}\right).
 \end{aligned}
\]
To derive these updates, for $\alpha=4$ the velocity recurrence is
$v_{n+1}=[(n-1)/(n+3)]v_n+sc$.
The particular solution $v_n=A(n+3)$ satisfies it because $sc=5A$.
Subtracting this solution gives a homogeneous recurrence, and therefore
\begin{align*}
 v_k-A(k+3)
 &=(v_m-A(m+3))\prod_{i=m}^{k-1}\frac{i-1}{i+3}\\
 &=(v_m-A(m+3))
   \frac{(m-1)m(m+1)(m+2)}{(k-1)k(k+1)(k+2)}.
\end{align*}
For the position, $x_k=x_m-\sum_{j=m+1}^k v_j$ and
\begin{align*}
 \sum_{j=m+1}^k(j+3)&=\frac12(k-m)(k+m+7),\\
 \frac1{(j-1)j(j+1)(j+2)}
 &=\frac13\left(\frac1{(j-1)j(j+1)}
                 -\frac1{j(j+1)(j+2)}\right).
\end{align*}
The latter identity telescopes when summed over $m+1\le j\le k$,
giving the displayed position formula. At the initial index $m=1$,
the zero momentum coefficient instead gives $v_2=sc$ and
$x_2=x_1-sc$ directly.
Each batch starts from the current forward state and is accepted only
when its endpoint queries lie in the same constant part of the fixed
spatial derivative. Positivity makes the intervening queries monotone,
so they remain in that interval. Steps near joins are evaluated
individually. This computes the discrete orbit without replacing it
by a continuous-time approximation or resetting it to inverse positions.

The calculation through $2N_5$ uses $20{,}894$ decimal digits, $883$
batches, and $57$ individual steps. A repeat with $100$ additional
digits leaves the displayed values unchanged. A separate check through
iteration $5000$ compares batching with individual oracle steps; their
state difference is below $10^{-179}$. Six retained blocks suffice for
the spatial oracle and integration; the relative omitted-area bound is
below $10^{-128393}$. All tail and precision checks are recorded with
the code.

\Cref{fig:no-gain-experiment} uses $n/N_j$ to magnify each episode
and $G(N_j)$ to compress the global time scale. The five selected ratios
$N_j^2G(N_j)\widehat F(x_{N_j})/K$ are approximately
$200$, $394$, $590$, $786$, and $983$; every displayed window crosses
the reference from below. These are five observed episodes of the
same loss. The bound $2j$ above establishes their infinite continuation.

\begin{figure}[!htbp]
\centering
\includegraphics[width=\linewidth]{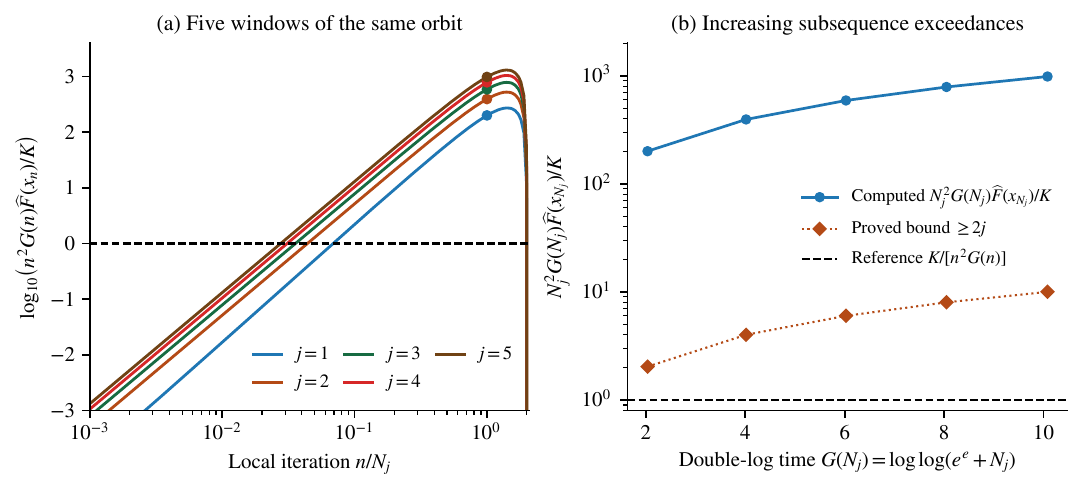}
\caption{Batched discrete CD on one fixed loss with
$G(n)=\log\log(e^e+n)$. Left: five windows of the same orbit, rescaled
by $N_j$; positive values mean that the error exceeds $K/[n^2G(n)]$.
Dots mark $n=N_j$. Right: selected ratios on a double-logarithmic time
axis, with their proved lower bounds. Connections between selected
points are visual guides. The dashed reference appears at $0$ on the
left because of the logarithm, and at $1$ on the right.}
\label{fig:no-gain-experiment}
\end{figure}

\subsection{Theorem 2: attaining the predicted power constant}

\begin{figure}[!htbp]
\centering
\includegraphics[width=\linewidth]{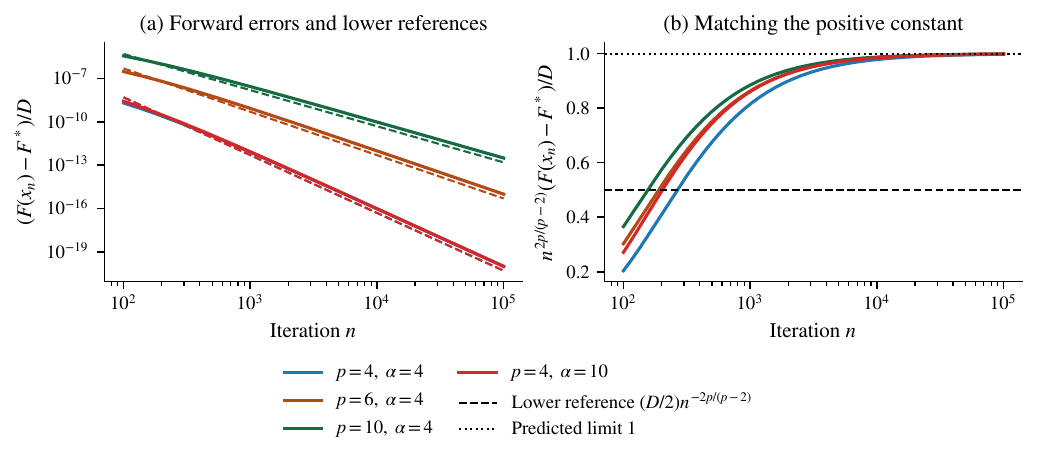}
\caption{Forward CD on four fixed power witnesses. Left: computed
errors divided by the predicted $D$, together with dashed references
$\tfrac12n^{-2p/(p-2)}$. Right: the normalized errors approach the
predicted limit $1$ and exceed $1/2$ throughout the reported observed
tails. The same normalizations are used for the reference lines in both
panels. These simulations evaluate the constructed interpolants directly.}
\label{fig:power-experiment}
\end{figure}

We use $g_n=(n+32)^{-\rho}$ with $\rho=2+2/(p-2)$.
The cases $p=4,6,10$ at $\alpha=4$ lie in the flatness regime.
An additional case $p=4$, $\alpha=10$ also satisfies the growth-only
damping condition in part~(ii); the constructed witness itself still
has $H(p)$. Each forward run reaches $n=10^5$. Objective integration
uses five million nodes, giving a maximum relative omitted-area bound
of $5.67\times10^{-4}$ across the four runs.

The normalization uses the analytic constant from
\cref{lem:power-leading}, including the spatial rescaling:
\[
 D=\lambda^2\frac{s}{2(\rho-1)(\alpha+1-\rho)}.
\]
No constant is fitted to the observed curve.

\Cref{fig:power-experiment} compares the computed errors with
$(D/2)n^{-2p/(p-2)}$, an eventual lower reference implied by the
positive asymptotic constant. The final normalized errors
$n^{2p/(p-2)}\widehat F(x_n)/D$ are respectively
$0.997791$, $0.998458$, $0.998677$, and $0.998471$.
In the tested ranges, this ratio stays above $1/2$ from indices
$267$, $190$, $156$, and $204$ onward, respectively.
The proof establishes the eventual lower bound for the full infinite
orbit; the runs show its finite-horizon manifestation.

\FloatBarrier

\section{Guide to the supplementary materials}\label{app:supplement}

The supplementary materials contain the Lean~4 development, construction
figure generators, and forward CD experiments. This appendix describes
their organization, gives reproduction commands, and specifies which
claims are formally checked and which are illustrated numerically.
All paths below are relative to the root of the accompanying materials.

\subsection{Contents and organization}

\begin{itemize}
\item \texttt{nag\_cd\_rate\_nonimprovability/} contains the Lean~4
sources, pinned toolchain and dependencies, and the public axiom audit.
Its theorem declarations and verification scope are described below.
\item \texttt{figures/} contains the generators for the two construction
schematics and the concrete instances in
Appendix~\ref{app:construction-instances}, together with generated plots
and numerical checks. The schematics explain the construction steps;
the concrete instances record recurrence checks and integration-tail
bounds. The proof dependency diagram is supplied as a separate TikZ source.
\item \texttt{experiments/fixed\_loss\_lower\_bounds/} contains the
fixed-loss gradient oracles, forward CD implementations, sampled traces,
plots, and numerical checks used in \cref{sec:experiments} and
Appendix~\ref{app:numerical-witnesses}. The Python runner
\texttt{reproduce.py} invokes \texttt{run\_cd.c} for individual CD steps;
\texttt{slow\_gain.py} computes the double-logarithmic-gain experiment
using arbitrary-precision batched updates. The script
\texttt{plot\_main.py} assembles the main-text experiment figure from
the saved data.
\end{itemize}
The accompanying README files describe the inputs, outputs, dependencies,
and recorded checks for each component.

\subsection{Scope of the Lean~4 formalization}\label{app:formal}

The Lean~4 development formalizes the no-universal-gain conclusion,
including every prescribed initial distance $R>0$, and the sharp
power-rate comparison.

The following declarations belong to the namespace
\texttt{NAGPointwiseRates} and support the corresponding conclusions of
\cref{thm:no-gain,thm:power}.
\begin{center}
\begin{tabular}{@{}p{.28\linewidth}>{\raggedright\arraybackslash}p{.68\linewidth}@{}}
\toprule
Conclusion & Lean~4 declaration \\
\midrule
Scalar no-gain with prescribed $R$ & \texttt{nagCD\_noPointwiseUniversalGain\_\allowbreak dimOne\_\allowbreak prescribedRadius} \\
Finite-dimensional no-gain with prescribed $R$ & \texttt{nagCD\_noPointwiseUniversalGain\_\allowbreak prescribedRadius} \\
Flatness--growth upper rate & \texttt{sharpPowerBoundary\_flatnessGrowth} \\
Growth-only upper rate & \texttt{sharpPowerBoundary\_growthOnly} \\
Fixed power witness & \texttt{sharpPowerBoundary\_noFurtherGain} \\
\bottomrule
\end{tabular}
\end{center}
The manuscript uses an affine block interpolant
for the no-gain construction and a shifted forcing sequence for the power example;
the formal development uses its own interpolants and a spliced exact power
tail. The correspondence is at the level of the specified theorem conclusions,
not an identification of these particular witness functions.
The intermediate lemmas and dependency diagram in the proof appendices organize
the manuscript proofs; their presence does not assert separate formal
verification of each displayed calculation.

\subsection{Rebuilding and auditing the Lean~4 development}

The package pins both Lean~4 and mathlib to version \texttt{v4.32.0}.
With the Lean~4 toolchain manager installed, run the following commands
from the root of the accompanying materials:
\begin{verbatim}
cd nag_cd_rate_nonimprovability
lake exe cache get
lake build
lake env lean AxiomAudit.lean
\end{verbatim}
The public audit reports only \texttt{propext},
\texttt{Classical.choice}, and \texttt{Quot.sound} for its 16 declarations.
Formal checking applies to the listed declarations.
The remaining manuscript arguments are supplied as mathematical proofs;
they are not claimed to have been checked by Lean~4.

\subsection{Reproducing the figures and numerical runs}

From the root of the accompanying materials, the following commands
create a Python environment, install the pinned numerical dependencies,
regenerate the construction figures, and rerun the forward experiments:
\begin{verbatim}
python3 -m venv .venv
.venv/bin/pip install -r \
  experiments/fixed_loss_lower_bounds/slow_gain_requirements.txt
.venv/bin/python figures/generate_constructions.py
make experiments PYTHON=.venv/bin/python
\end{verbatim}
The dependency file includes the requirements for both experiment engines
and the construction figures. The individual-step runner uses a C compiler
and requires \texttt{long double} with at least 64 mantissa bits; the
dependencies include a compiler fallback. The batched experiment repeats
its calculation at higher precision and compares a shorter run against
individual CD updates. The scripts record precision comparisons, source
information, and analytic bounds on omitted integration tails.

To redraw \cref{fig:main-experiments} from the supplied traces without
rerunning the simulations, use
\begin{verbatim}
make experiment-overview PYTHON=.venv/bin/python
\end{verbatim}
These numerical checks concern the reported finite runs. The analytic
tail bounds do not enclose every floating-point rounding error; the
infinite-time and arbitrary-gain conclusions are established by the
mathematical proofs.

\end{document}